\documentclass{amsart}
\usepackage{graphicx}
\usepackage{amssymb,latexsym,wasysym}
\usepackage{accents}
\usepackage{amsfonts,amsmath,amsthm}
\usepackage{color}
\usepackage{enumerate}
\usepackage[dvipsnames]{xcolor}
\usepackage{url}
\usepackage{tcolorbox}
\usepackage[colorlinks=true,
linkcolor=violet,
urlcolor=Aquamarine,
citecolor=YellowOrange]{hyperref}
\newtheorem{thm}{Theorem}[section]
\newtheorem{cor}[thm]{Corollary}
\newtheorem{lem}[thm]{Lemma}
\newtheorem{prop}[thm]{Proposition}
\newtheorem{defn}[thm]{Definition}
\newtheorem{rem}[thm]{Remark}

\newcommand{\norm}[1]{\left\Vert#1\right\Vert}
\newcommand{\abs}[1]{\left\vert#1\right\vert}
\newcommand{\set}[1]{\left\{#1\right\}}
\newcommand{\Real}{\mathbb R}

\newcommand{\pfrac}[2]{\frac{\partial #1}{\partial #2}}

\newcommand{\bx}{\mathrm{x}}

\numberwithin{equation}{section}

\begin{document}
\title{Curvature at the infinity of asymptotically flat Einstein manifold}

\author{Bing Wang}
\author{Hao Yin}

\address{Bing Wang,  Institute of Geometry and Physics, and School of Mathematical Sciences,
University of Science and Technology of China, Hefei 230026, China;
Hefei National Laboratory, Hefei 230088, China}
\email{topspin@ustc.edu.cn}
\address{Hao Yin,  School of Mathematical Sciences,
University of Science and Technology of China, Hefei 230026, China}
\email{haoyin@ustc.edu.cn }

\begin{abstract}
	In this paper, for $m\geq4$, we construct coordinates at the infinity of an asymptotically flat end of a Ricci-flat manifold $(M^m,g)$ as long as the $L^{m/2}$ norm of the curvature is finite in this end. As applications, we can define a Weyl tensor at the infinity of $M$ and a version of renormalized volume following Biquard and Hein to ALE Ricci-flat manifolds/orbifolds of any dimension.
\end{abstract}

\maketitle
\section{Introduction}
\label{sec:intro}

Suppose that $(M,g)$ is an $m$-dimensional Riemannian manifold ($m\geq 4$) and $\Omega$ is an end of $M$ satisfying
\begin{enumerate}
	\item $(M,g)$ is Ricci-flat on $\Omega$;
	\item $Vol(B(o,r)\cap \Omega)\geq C r^m$, for some fixed $o\in M$ and $C>0$ and any large $r>0$;
	\item $\int_\Omega \abs{Rm}^{m/2} dV_g<\infty $.
\end{enumerate}
It was proved in a famous paper of Bando, Kasue and Nakajima \cite{bando1989on} that this end is ALE of order $\tau=m-1$ in the sense that there exist a compact subset $K\subset M$, $R>0$, a finite subgroup $\Gamma\subset O(m)$ acting freely on $\Real^m \setminus B(0,R)$, a diffeomorphism $\mathcal L$ from $\Omega\setminus K$ to $(\Real^m \setminus B(0,R))/\Gamma$ such that if
\[
	\tilde{g}:= (\mathcal L^{-1}\circ {\rm proj})^* g
\]
where ${\rm proj}$ is the projection from $\Real^m\setminus B(0,R)$ to $(\Real^m\setminus B(0,R))/\Gamma$, then
\[
	\tilde{g}_{ij}(z)= \delta_{ij} + O(\abs{z}^{-\tau})
\]
and for any positive integer $k$
\begin{equation}
	\label{eqn:der}
	\partial^{(k)} \tilde{g}_{ij}(z)=O(\abs{z}^{-\tau-k}).
\end{equation}
The estimates on the derivatives follow from the Ricci-flat condition and the fact that the coordinate obtained in \cite{bando1989on} is a harmonic coordinate system. 

Moreover, in the case $m=4$, or $(M,g)$ is K\"ahler, this $\tau$ can be improved to be $m$ (see Theorem 1.5 of \cite{bando1989on}). Later, it was observed by Cheeger and Tian \cite{cheeger1994} that as long as $\Gamma$ is not trivial, $\tau$ can be taken to be $m$ in general. Kr\"oncke and Szab\'o recently \cite{kroencke2022optimal} pointed out a gap in \cite{cheeger1994} and fixed it as a corollary of their main theorems. It is worth noting that both \cite{cheeger1994} and \cite{kroencke2022optimal} discuss the more general setting of asymptotically conical end.

While $\tau=m$ is optimal, various authors continue to improve our understanding towards the regularity of metric at the infinity. For example, Chen \cite{chen2020expansion} proved a full expansion of $\tilde{g}_{ij}$ in the coordinate constructed in \cite{bando1989on}. Recently, more precise expansion for harmonic functions and harmonic one-forms was obtained by Chen and Yan \cite{chen2025harmonicformsalericciflat}.

For $m=4$, Biquard and Hein conducted a detailed analysis of the leading term $h_0$ in the expansion of $\tilde{g}_{ij}$. In particular, in Section 2 of \cite{biquard2023renormalized}, it was proved that $h_0$ is the sum of harmonic gauge terms and Kronheimer terms. While the former could be removed by a further coordinate change, the latter was shown to be exactly the leading term of the metric expansion for the family of ALE K\"ahler Ricci-flat four-manifolds due to Kronheimer. Although not explicitly mentioned, the computation in Proposition 2.6 in \cite{biquard2023renormalized} actually established an isomorphism between the space of Kronheimer terms and the space of Weyl tensors on $\Real^4$ (see also Appendix A of \cite{ozuch2024}). 

\subsection{The main results}

In this paper, we revisit the problem of establishing optimal decay rates and characterizing the leading terms in general dimensions. We need Definition \ref{defn:O}.
\begin{defn}
	\label{defn:O}
	For a function $f$ defined on $\Real^m\setminus  B(0,R)$, it is said to be in $O(\abs{x}^{-\tau})$ if there exist constants $C(k)$ such that
	\[
		\abs{\partial^{(k)} f}(x)\leq C(k) \abs{x}^{-\tau-k} \qquad \text{for} \quad k=0,1,2,\cdots 
	\]
	on $\Real^m \setminus B(0,R)$.
\end{defn}
Our first result is
\begin{thm}
	\label{thm:main1}
	Let $(M,g)$ be an $m$-dimensional Riemannian manifold and $\Omega$ is an end satisfying
	\begin{enumerate}
		\item $(M,g)$ is Ricci-flat on $\Omega$;
		\item $Vol(B(o,r)\cap \Omega)\geq C r^m$, for some fixed $o\in M$ and $C>0$ and any large $r>0$;
		\item $\int_\Omega \abs{Rm}^{m/2} dV_g<\infty $.
	\end{enumerate}
	Then there is a compact set $K$ in $M$ such that there is a diffeomorphism $\Psi$ from the universal cover of $\Omega \setminus K$ to $\Real^m \setminus B$ that gives a coordinate system satisfying

	(I) it is Bianchi, in the sense that if $g_e$ is the flat metric given by the coordinates, then
	\[
		B_{g_e}g := \delta_{g_e}g - \frac{1}{2} d({\rm Tr}_{g_e} g)=0.
	\]
	Moreover, this is equivalent to saying that $\Psi^{-1}$ is a harmonic map from $g_e$ to $g$.

	(II) the metric tensor has an expansion
	\begin{equation}
		\label{eqn:expansion}
		g_{ij}= \delta_{ij} + W_{iklj}\frac{x_k x_l}{\abs{x}^{{m}+2}} + O(\abs{x}^{-m-1}).
	\end{equation}
	Here $W_{iklj}$ is a Weyl tensor.
\end{thm}

\begin{rem}
	Let $\Gamma$ be the fundamental group of $\Omega\setminus K$. In previously mentioned results such as \cite{cheeger1994}, \cite{kroencke2022optimal}, it is necessary to assume $\Gamma\ne \set{1}$. This is not needed for Theorem \ref{thm:main1}. The key point is that instead of focusing on one coordinate, we use an intertwined process of finding new coordinates and improving the regularity of the metric. Indeed, in the rest of the paper, we always assume that $\Omega\setminus K$ is simply connected. In case $\Gamma\ne \set{1}$, it is easy to pass to the universal cover and work with $\Gamma$-equivariant spaces in the proofs.
\end{rem}

\begin{rem}
	There is a harmonic coordinate version of Theorem \ref{thm:main1}, with (I) replaced by: (I') the components of $\Psi$ are harmonic functions. This could be  proved similarly. Instead of repeating, we will indicate at various places how the proofs can be modified for this purpose.
\end{rem}

Given Theorem \ref{thm:main1}, it is natural to ask if the Weyl tensor $W$ is independent of the choice of coordinate. Or equivalently, is it an intrinsic geometric property of an end of manifold satisfying the assumptions of Theorem \ref{thm:main1}? Theorem \ref{thm:main2} provides a positive answer.

\begin{thm}
	\label{thm:main2} Let $(M,g)$ and $\Omega$ be as in Theorem \ref{thm:main1}. Assume that there are two coordinate systems $(x_k)$ and $(y_k)$ defined on (the universal cover of) $\Omega\setminus K$ such that in each of these coordinate systems, we have Weyl-type tensors $W$ and $\tilde{W}$ respectively satisfying  
	\[
		g= (\delta_{ij}+W_{iabj}\frac{x_ax_b}{\abs{x}^{m+2}}+O(\abs{x}^{-(m+1)})) dx_i\otimes dx_j	
	\]
	and
	\[
		g= (\delta_{ij}+\tilde{W}_{iabj}\frac{y_ay_b}{\abs{y}^{m+2}}+O(\abs{y}^{-(m+1)})) dy_i\otimes dy_j.
	\]
	Then there exists an $A\in O(m)$ such that
	\[
		W_{ijkl}=\tilde{W}_{abcd}A_{ai}A_{bj}A_{ck}A_{dl}.
	\]
\end{thm}

\begin{rem}
	In Definition \ref{defn:O}, we have assumed a bound for derivatives of any order. This is not necessary for the validity of Theorem \ref{thm:main2}. We have chosen to keep the exposition simple.
\end{rem}

The coordinate in Theorem \ref{thm:main1} in spite of possessing good properties is not unique (up to rotation). Hence, it is natural to ask the following question (with the exponential map and the normal coordinate in mind): is it possible to define a better coordinate that is unique (up to rotation) and whose metric expansion contains more geometric information?

We then move on to discuss the renormalized volume. On noncompact manifolds, the total volume is often infinity. However, under reasonable assumptions on the geometry of the manifolds, it is often useful to compare this infinity (more precisely the growth rate) with some standard example and the difference is usually called the renormalized volume. It was first defined for asymptotically hyperbolic Einstein metrics by Graham \cite{graham}. Later, Anderson \cite{anderson} found an explicit formula for the renormalized volume in this case. Brendle and Chodosh \cite{brendle} defined a renormalized volume in a different setting for asymptotically hyperbolic metrics. For the asymptotically Euclidean (or flat) case, the renormalized volume was defined by Biquard and Hein \cite{biquard2023renormalized} for ALE Ricci-flat four-manifolds/orbifolds. In this case, they proved that the renormalized volume is non-positive and established the rigidity for the equal case. 


In this paper, using the coordinates above, we generalize the definitions of Biquard and Hein to higher dimensions. First, note that Theorems \ref{thm:main1} and \ref{thm:main2} concern the geometry of an end. Therefore, we may extend their validity to spaces that are orbifolds rather than manifolds, provided that the end itself satisfies the assumptions. Specifically, we allow $M$ to be an orbifold with finitely many isolated singularities. At each singularity, a neighborhood is homeomorphic to the quotient of a ball in $\mathbb R^m$ by a finite subgroup of $SO(m)$ — a type of orbifold that arises naturally in the compactness theory of Einstein manifolds (see \cite{anderson, bando1990bubbling}). Under this extension, both Theorem \ref{thm:main1} and Theorem \ref{thm:main2} remain valid. Let $x$ denote the coordinate defined on $\mathbb R^m \setminus B$ as given in Theorem \ref{thm:main1}.
Set
$$
B_r=M \setminus \set{\abs{x}>r} \quad \text{and} \quad \tilde{B}_r = \set{x\in \Real^m | \, \abs{x}<r}/\Gamma.
$$
The renormalized volume $\mathcal V$ is defined to be the limit
\begin{equation}
	\label{eqn:v}
	\lim_{r\to \infty } V_g(B_r)- V_{g_e}(\tilde{B}_r).
\end{equation}
Following the discussions of \cite{biquard2023renormalized}, we have
\begin{thm}
	\label{thm:volume}
	Let $(M,g)$ be $m$-dimensional ALE Ricci-flat manifold/orbifold such that the assumptions of Theorem \ref{thm:main1} hold for its end $\Omega$.  Using the above notations, the limit in \eqref{eqn:v} exists and is independent of the choice of the coordinate that satisfies the assumptions of Theorem \ref{thm:main2}. Moreover, $\mathcal V\leq 0$. The equality holds if and only if $(M,g)$ is $\Real^m/\Gamma$.
\end{thm}

In the case of $m=4$, the renormalized volume $\mathcal V$ coincides with the one defined in \cite{biquard2023renormalized}. In \cite{biquard2023renormalized}, the uniqueness of CMC foliation (see \cite{chodosh2017}) is used in order to justify that $\mathcal V$ is an intrinsic geometric quantity. Here, we take a different approach. Our proof relies on the fact that any two coordinates satisfying the assumptions of Theorem \ref{thm:main2} are the same up to high order.

\subsection{Idea of the proofs}

To illustrate the idea used in the proof of Theorem \ref{thm:main1}, we start from the well-known expansion of any Riemannian metric $g_{ij}$ in its normal coordinate system,
\[
	g_{ij}=\delta_{ij}-\frac{1}{3}R_{iabj}x_ax_b + O(\abs{x}^3).
\]
This formula is discussed in many textbooks on Riemannian geometry. The popular proof uses Jacobi fields. There is another direct proof due to Riemann (as explained by Spivak in \cite{spivak1979}, see Volume 2, Chapter 4). There is also an interesting and informative discussion of the history of this formula available at StackExchange \cite{web}. Among the answers posted there, there is one due to Bryant that is very relevant to the proofs in this paper, which we summarize briefly as follows.

Since $g_{ij}$ is smooth, one starts with the Taylor expansion
\begin{equation}
	\label{eqn:temp}
	g_{ij}=a_{ij}+b_{ij,k}x_k+ c_{ij,kl}x_kx_l + \cdots 
\end{equation}
It is easy to see that we can use a linear coordinate change to make $a_{ij}$ equal to $\delta_{ij}$, in which sense \eqref{eqn:temp} is simplified. The idea is that we may continue to do so and look for better and better coordinate in which $g_{ij}$ becomes as "simple" as possible. By considering a coordinate change
\[
	\varphi(x)_i = x_i + A_{jk,i}x_jx_k,
\]
we can eliminate the linear term $b_{ij,k}x_k$ if $A_{jk,i}$ is chosen correctly (by solving some linear equations). Finally, one may consider one more coordinate change
\[
	\varphi(x)_i=x_i + B_{jkl,i}x_jx_kx_l
\]
to reduce the $c_{ij,kl}$ to a curvature tensor. To the best of our knowledge, we do not know a complete proof in the literature. Hence, we provide the full details of this proof in Section \ref{sec:smooth} with the linear algebra formulas proved in Section \ref{sec:alg}.

A part of Theorem \ref{thm:main1} amounts to an adaptation of the above argument to the infinity of an asymptotically flat  Ricci-flat metric. We take an intertwined argument of algebra (to eliminate/simplify the coefficients in the expansion) and analysis (to produce the next nontrivial leading term in the expansion using Einstein equation). More precisely, we use \cite{bando1989on} as our starting point (hence our argument is then independent of later development), i.e. a harmonic coordinate at infinity in which (by elliptic regularity of the Einstein equation) the metric is
\[
	g_{ij}=\delta_{ij}+A_{ijk}\frac{x_k}{\abs{x}^m} + O(\abs{x}^{-m}).
\]
Some algebraic lemma allows us to find a new coordinate in which $A_{ijk}\equiv 0$. This is a feature of the present argument in the sense that we do not assume $\Gamma\ne \set{1}$ to show the vanishing of $A_{ijk}$, instead, we find a new coordinate. Then, we solve the harmonic map equation to turn the new coordinate into a Bianchi coordinate while keeping the expansion. Then elliptic regularity of the Einstein equation again gives an expansion
\[
	g_{ij}=\delta_{ij}+ A_{ijkl}\frac{x_kx_l}{\abs{x}^{m+2}}+ O(\abs{x}^{-m-1}).
\]
Another algebraic argument implies the existence of a new coordinate in which this $A_{ijkl}$ is reduced to a Weyl tensor. Finally, we solve the harmonic map equation again to find the coordinate claimed in Theorem \ref{thm:main1}.

Now, we explain the proof of Theorem \ref{thm:main2}. The identity map is certainly a harmonic map. If we use coordinate $x$ in the domain and coordinate $y$ in the target, the identity map is given by $\varphi(x)=y$ for some smooth function $\varphi$. Given that $g_{ij}$ is approximately $\delta_{ij}$ in both $x$-coordinate and $y$-coordinate, it is obvious that when $x$ goes to infinity, $\varphi$ gets closer and closer to an orthogonal transformation plus a translation. The key to the proof is to develop the regularity of $\varphi$ using the harmonic map equation and the assumed expansion of metric tensor in both coordinates. For simplicity, if we assume the orthogonal transformation and the translation at the infinity is trivial, the regularity analysis yields
\[
	\varphi(x)_i=x_i + B_{ij}\frac{x_j}{\abs{x}^m}+ \cdots 
\]
for some $2$-tensor $B\in (\Real^m)^{\otimes 2}$.
Formal computation shows that it is this $2$-tensor $B_{ij}$ that is responsible for transforming $W$ to $\tilde{W}$. However, some linear algebra result (see Lemma \ref{lem:findB}) implies that $W=\tilde{W}$ as claimed by the theorem.

Given the existence and the ``almost" uniqueness of a good coordinate at the infinity, it is natural to use it to define various geometric quantities. Besides the Weyl tensor given in Theorem \ref{thm:main2}, we follow Biquard and Hein \cite{biquard2023renormalized} for the definition of the renormalized volume \eqref{eqn:v}. Our uniqueness result implies that it is an intrinsic geometric quantity. The rest of Theorem \ref{thm:volume} is proved using a result of \cite{ros1987} (as suggested by Remark 3.1 of \cite{biquard2023renormalized}).

The rest of the paper is organized as follows. In Section \ref{sec:alg}, we discuss some linear algebra results needed later. In Section \ref{sec:analysis}, we prove a theorem about the existence of Bianchi coordinate. In Section \ref{sec:smooth}, we give precise details to the local result outlined by Bryant that serves as a precursor to Theorem \ref{thm:main1}, and strengthen it a little by asking the coordinate to be Bianchi. The result in this section is not logically necessary for the proof of the main theorems in this paper. We simply hope they are interesting in their own right. In Section \ref{sec:infinity}, we prove the two main theorems. Finally, in Section \ref{sec:application}, we discuss the concept of renormalized volume and show that in the special case of dimension four, it agrees with what is known in \cite{biquard2023renormalized}.

{\bf Acknowledgments.} Bing Wang is partially supported by YSBR-001, NSFC-12431003, and a research fund from USTC. Hao Yin is partially supported by NSFC-12431003. The authors would like to thank the anonymous referees and Professor Li Yu for their valuable comments, which helped improve the manuscript.

\section{Some algebraic facts}
\label{sec:alg}

In this section, we define some maps between tensor spaces and discuss a few properties that are useful to the proof of our main results. The maps defined in Section \ref{subsec:maps} and their properties may look strange at first, but they will be shown to be useful later.

\subsection{Notations}

To start, we fix notations. Throughout this paper, we are interested in various tensors on $\Real^m$. With the standard inner product of $\Real^m$ in mind, we make no distinction between covariant and contravariant tensors. If $A$ is a $4$-tensor, its components are denoted by $A_{i_1i_2i_3i_4}$. For simplicity, we write $A_{(1234)}$ for $A_{i_1i_2i_3i_4}$. We denote by $[A_{(1234)}]$ the tensor whose components are $A_{(1234)}$. Please note that using this notation, $[A_{(2134)}]$ is different from $[A_{(1234)}]$. The same convention applies to $2$-tensors and $3$-tensors.

We use $\delta=[\delta_{(12)}]$ for the Kronecker delta. We also assume the summation convention that repeated indices are summed over. We use ${\rm Tr}_{(13)}$ for the trace over the first and the third indices of a tensor. For example, if $A$ is a $4$-tensor, by taking ${\rm Tr}_{(13)}$, we obtain a $2$-tensor.
\[
	{\rm Tr}_{(13)}([A_{(1234)}])=[A_{(a1a2)}].
\]
We also use $S_{(12)}$ for the symmetrization of the first and second indices of a tensor. More precisely,
\[
	S_{(12)}([A_{(123)}])= \frac{1}{2}[A_{(123)}+A_{(213)}].
\]
$S_{(234)}$ is understood similarly.

$S^k(\Real^m)$ is the space of symmetric $k$-tensors. $S^2_0(\Real^m)$ is the traceless symmetric $2$-tensors and $\Lambda^2(\Real^m)$ is the skew-symmetric $2$-tensors. These spaces arise naturally in our paper. For example, when we consider the Taylor expansion of a metric $g_{ij}$, the linear term is $a_{ijk}x_k$ where $[a_{(123)}]$ is a tensor in $S^2(\Real^m)\otimes \Real^m$. 

Finally, $\mathcal C$ is the space of curvature tensors and $\mathcal W$ is the space of Weyl tensors. We define a map $s$ on $\mathcal C$ by
\[
	s([R_{(1234)}])=\frac{1}{2} ([R_{(1342)}+ R_{(1432)}])\in S^2(\Real^m)\otimes S^2(\Real^m) \qquad \text{} \quad \forall R\in \mathcal C.
\]
It is well known that $s$ is injective on $\mathcal C$. We denote $s(\mathcal C)$ by $\widetilde{\mathcal C}$ and $s(\mathcal W)$ by $\widetilde{\mathcal W}$.

\begin{rem}
	The map $s$ is related to the so-called curvature operator of the second kind, which we refer to Section 2.3 of \cite{bourguignon1978} for its definition. 
\end{rem}

\begin{rem}
	\label{rem:curv_sym} As a corollary of the first Bianchi identity, for any tensor $A$ in $\widetilde{\mathcal C}$, we have
	\[
		A_{(1234)}+A_{(1342)}+A_{(1423)}=0.
	\]
\end{rem}

\subsection{Maps between tensors}
\label{subsec:maps}

Using these notations, we write down a few important maps between tensors that will be used in later proofs.\\

\begin{enumerate}
	\item ${\mathbf R}: S^2(\Real^m) \otimes S^2(\Real^m)\to \mathcal C$ defined by
		\begin{equation}
			\label{eqn:R}
			{\mathbf R}(A)= [ A_{(2413)}-A_{(2314)}-A_{(1423)}+A_{(1324)}].
		\end{equation}
		For the meaning of this map, we notice that if the metric is $g_{(12)}=\delta_{(12)}+A_{(1234)}x_{(3)}x_{(4)}$ for $A\in S^2(\Real^m) \otimes S^2(\Real^m)$, then the curvature tensor at the origin is ${\mathbf R}(A)$.
	\item $\tilde{\mathbf R}: S^2(\Real^m) \otimes S^2(\Real^m)\to \widetilde{\mathcal C}$ defined by
		\[
			\tilde{\mathbf R}(A)=  -\frac{1}{6} [{\mathbf R}(A)_{(1342)}+{\mathbf R}(A)_{(1432)}].
		\]
		Given $g$ as above, this is the tensor that appears as the second order term in the expansion of $g$ in the normal coordinate system. 

	\item $\Psi_1: S^2(\Real^m)\otimes \Real^m\to S^2(\Real^m)\otimes \Real^m$ defined by
		\[
			\Psi_1(B)=-2 [B_{(312)}+B_{(321)}].
		\]
		This is related to the coordinate change of the linear term in the expansion of $g_{ij}$. 
	\item $\Psi_2: \Real^m \otimes S^3(\Real^m)\to S^2(\Real^m)\otimes S^2(\Real^m)$ defined by
		\begin{equation*}
			\Psi_2(B)= -3[B_{(1234)}+B_{(2134)}].
		\end{equation*}
		This is related to the coordinate change of the quadratic term in the expansion of $g_{ij}$. 
	\item $\Psi_3: \Real^m\to S^2(\Real^m)\otimes \Real^m$ defined by
		\[
			\Psi_3(B)=-(m-2)[ \delta_{(13)}B_{(2)}+\delta_{(23)}B_{(1)}].
		\]
		This is related to the coordinate change of the $-(m-1)$ order term in the expansion of an asymptotically flat metric at infinity.
	\item $\Psi_4: (\Real^m)^{\otimes 2}\to S^2(\Real^m)\otimes S^2(\Real^m)$ defined by
		\[
			\Psi_4(B)= [B_{(12)}\delta_{(34)}+B_{(21)}\delta_{(34)}- \frac{m}{2}(B_{(13)}\delta_{(24)}+B_{(14)}\delta_{(23)}+B_{(23)}\delta_{(14)}+B_{(24)}\delta_{(13)})].
		\]
		This is related to the coordinate change of the $-m$ order term in the expansion of an asymptotically flat metric at infinity.

	\item The Bianchi operator acting formally on the $-(m-1)$ order term of the metric expansion motivates
		\[
			{\mathbf B}_1: S^2(\Real^m)\otimes \Real^m \to S^2(\Real^m)  \otimes \Real^m
		\]
		\begin{equation*}
			{\mathbf B}_1 (A)= [2A_{(3aa)}\delta_{(12)} - m(A_{(312)}+A_{(321)}) - A_{(aa3)}\delta_{(12)}+\frac{m}{2}(A_{(aa2)}\delta_{(13)}+A_{(aa1)}\delta_{(23)})].
		\end{equation*}
	\item The Bianchi operator acting formally on the $-m$ order term of the metric expansion motivates
		\[
			{\mathbf B}_2: S^2(\Real^m)\otimes S^2(\Real^m)\to \Real^m \otimes S^3(\Real^m)
		\]
		\[
			{\mathbf B}_2(A)= S_{(234)}[4A_{(a1a4)}\delta_{(23)}-2(m+2)A_{(3124)}-2A_{(aa14)}\delta_{(23)}+(m+2)A_{(aa24)}\delta_{(13)}].
		\]
\end{enumerate}

\subsection{Decomposition of tensor spaces}

In this subsection, we collect a few properties for the maps defined in Section \ref{subsec:maps}. Those maps can be regarded as morphisms of representations of $O(m)$. Hence, they may be understood via a decomposition of the domains and the images into irreducible representations. However, we refrain from using any abstract theory, instead we prove only results that are needed in later proofs using elementary and explicit methods.

\begin{prop}
	\label{prop:psi1}
	The map $\Psi_1: S^2(\Real^m)\otimes \Real^m\to S^2(\Real^m)\otimes \Real^m$ is an isomorphism.
\end{prop}
\begin{proof}
	$\Psi_1$ is a linear map between vector spaces of the same dimension. Hence it suffices to show that $\Psi_1$ is injective. Namely, for any $[B_{(123)}]\in S^2(\Real^m)\otimes \Real^m$ satisfying 
	\[
		B_{(312)}= -B_{(321)},
	\]
	we need to show that $[B_{(123)}]=0$. In fact, since $B$ is symmetric in the first two indices and skew-symmetric in the last two indices, we obtain
	\begin{eqnarray*}
		B_{(123)} &=& B_{(213)} = -B_{(231)}\\
			  &=& -B_{(321)} = B_{(312)} = B_{(132)} \\
			  &=& -B_{(123)},
	\end{eqnarray*}
	which implies that $B=0$.
\end{proof}

\begin{prop}
	\label{prop:psi2}
	There is the direct sum decomposition
	\[
		S^2(\Real^m)\otimes S^2(\Real^m) = \widetilde{\mathcal C} \oplus  \Psi_2(\Real^m \otimes S^3(\Real^m)).
	\]
	More precisely, for any $A\in S^2(\Real^m) \otimes S^2(\Real^m)$,
	\[
		A= \tilde{\mathbf R}(A)+ \Psi_2(B)
	\]
	where $B\in \Real^m \otimes S^3(\Real^m)$ is defined by
	\begin{equation}
		\label{eqn:B}
		B= -\frac{1}{18}[2A_{(1234)}+2A_{(1324)}+2A_{(1423)}-A_{(3412)}-A_{(2314)}-A_{(2413)}].
	\end{equation}
\end{prop}
\begin{proof}
	The fact that $A= \tilde{\mathbf R}+\Psi_2(B)$ can be proved by straightforward computation. In fact, we have 
	\[
		\tilde{\mathbf R}(A)= -\frac{1}{6}[A_{(2314)}+A_{(2413)}-2A_{(3412)}-2A_{(1234)}+A_{(1432)}+A_{(1324)}]
	\]
	and 
	\[
		\Psi_2(B)=\frac{1}{6}\left[ 4A_{(1234)}-2A_{(3412)}+A_{(1324)}+A_{(1423)}+A_{(2314)}+A_{(2413)} \right],
	\]
	whose sum is exactly $A$.

	To see that this is a direct sum, assume that we have $C\in \widetilde{\mathcal C}$ and $B\in \Real^m \otimes S^3(\Real^m)$ such that
	\[
		C+ \Psi_2(B)=0.
	\]
	By Remark \ref{rem:curv_sym}, $C$ and hence $\Psi_2(B)$ satisfy the first Bianchi identity,
	\[
		(\Psi_2(B))_{(1234)}+(\Psi_2(B))_{(1342)}+(\Psi_2(B))_{(1423)}=0,	
	\]
	which is (using the definition of $\Psi_2$)
	\begin{equation}
		\label{eqn:temp1}
		3B_{(1234)}+ B_{(2134)}+B_{(3124)}+B_{(4123)}=0.
	\end{equation}
	Here in the above computation, we have used the symmetry of $B$ in its last three indices. We do a permutation of the four indices in the above equation and sum them up to see
	\[
		B_{(1234)}+ B_{(2134)}+B_{(3124)}+B_{(4123)}=0,
	\]
	which (together with \eqref{eqn:temp1}) implies that $B=0$. Therefore $C=0$ and the proof of the proposition is done.
\end{proof}

\begin{lem}\label{lem:decom_21}
	We have the following direct sum decomposition
	\[
		S^2(\Real^m)\otimes \Real^m= Y_1 \oplus Y_2\oplus H_{2,1}.
	\]
	Here
	\begin{eqnarray*}
		Y_1&=& \set{[ \delta_{(12)}B_{(3)}] |\, B\in \Real^m}\\
		Y_2&=& \set{[\delta_{(13)}B_{(2)}+\delta_{(23)}B_{(1)}]|\, B\in \Real^m}
	\end{eqnarray*}
	and $H_{2,1}$ is the set of harmonic tensors in $S^2(\Real^m)\otimes \Real^m$, in the sense that trace with respect to any pair of indices vanishes.
\end{lem}
\begin{proof}
	By definition, $Y_1$ and $Y_2$ are subspaces of $S^2(\Real^m)\otimes \Real^m$.
	For a tensor $A\in S^2(\Real^m)\otimes \Real^m$, recall that $T_{(12)}$ and $T_{(13)}$ are defined by
	\begin{eqnarray*}
		T_{(12)}A &=& [A_{(aa1)}]  \\
		T_{(13)}A &=& [A_{(a1a)}] .
	\end{eqnarray*}
	Given the symmetry of $A$, these are the only independent traces. By computation, we have	
	\begin{eqnarray*}
		T_{(12)}[\delta_{(12)}B_{(3)}] &=& m B \\
		T_{(13)}[\delta_{(12)}B_{(3)}] &=& B \\
		T_{(12)}[\delta_{(13)}B_{(2)}+ \delta_{(23)}B_{(1)}] &=& 2B \\
		T_{(13)}[\delta_{(13)}B_{(2)}+ \delta_{(23)}B_{(1)}] &=& (m+1)B.
	\end{eqnarray*}
	Since the matrix
	\[
		\left(
			\begin{array}{cc}
				m & 1\\
				2 & m+1
			\end{array}
		\right)
	\]
	is invertible ($m\ne 1,-2$), for any $A\in S^2(\Real^m)\otimes \Real^m$, we can find $B$ and $B'$ in $\Real^m$ such that
	\[
		A- [\delta_{(12)}B_{(3)}] - [\delta_{(13)}B'_{(2)}+ \delta_{(23)}B'_{(1)}]
	\]
	is totally traceless, i.e. a harmonic tensor. 

	To see that this is a direct sum decomposition, assume that we have $B,B'\in \Real^m$ and $A\in H_{2,1}$ such that
	$$
	[\delta_{(12)}B_{(3)}] + [\delta_{(13)}B'_{(2)}+\delta_{(23)}B'_{(1)}] + A =0.
	$$
	By taking trace, we find that $B=B'=0$ and hence $A=0$.
\end{proof}
Using Lemma \ref{lem:decom_21}, we can show
\begin{prop}
	\label{prop:psi3}
	\[
		{\rm Ker}({\mathbf B}_1) = {\rm Im}(\Psi_3).
	\]
\end{prop}
\begin{proof}
	By definition of $\Psi_3$, it is obvious that ${\rm Im}(\Psi_3)=Y_2$ in Lemma \ref{lem:decom_21}. We verify directly that $Y_2$ lies in the kernel of ${\mathbf B}_1$. To see this, for any $B\in \Real^m$,
	\begin{eqnarray*}
	&& {\mathbf B}_1( [\delta_{(13)}B_{(2)}+\delta_{(23)}B_{(1)}]) \\
	&=& \left[ 2(m+1)B_{(3)}\delta_{(12)}- m \delta_{(23)}B_{(1)}-m \delta_{(12)}B_{(3)} -m\delta_{(13)}B_{(2)}-m\delta_{(12)}B_{(3)} \right] \\
	&& + \left[ - 2B_{(3)}\delta_{(12)} + \frac{m}{2}\left( 2B_{(2)}\delta_{(13)}+ 2B_{(1)}\delta_{(23)} \right)\right] \\
	&=& 0.
	\end{eqnarray*}
	It remains to show that ${\mathbf B}_1$ is injective on $Y_1\oplus H_{2,1}$.

	For $B\in \Real^m$, we compute
	\begin{equation}
		\label{eqn:b1y1}
		{\mathbf B}_1 ([\delta_{(12)}B_{(3)}]) = (m-2)\left[ \frac{m}{2}(B_{(2)}\delta_{(13)}+B_{(1)}\delta_{(23)})-B_{(3)}\delta_{(12)} \right],
	\end{equation}
	which implies that ${\mathbf B}_1$ is injective in $Y_1$. (It suffices to consider $T_{(23)}$ of the right hand side of \eqref{eqn:b1y1}.)

	For $A\in H_{2,1}$, the definition of ${\mathbf B}_1$ is simplified
	\[
		{\mathbf B}_1(A)= -m[A_{(312)}+A_{(321)}].
	\]
	If $A\in H_{2,1}\cap {\rm Ker}({\mathbf B}_1)$, then it is symmetric in the first two indices and skew-symmetric in the last two indices, from which we conclude that $A=0$. We have proved this before, see the end of proof of Proposition \ref{prop:psi1}.
	In other words, ${\mathbf B}_1$ is also injective in $H_{2,1}$. 

	To finish the proof, it remains to show that if there are $B\in \Real^m$ and $A\in H_{2,1}$ satisfying
	\[
		{\mathbf B}_1([\delta_{(12)}B_{(3)}]+A)=0
	\]
	then both $B$ and $A$ are $0$. To see this, we notice that ${\mathbf B}_1(A)$ is traceless, while for 
	\[
		{\rm Tr}_{(13)}( {\mathbf B}_1([\delta_{(12)}B_{(3)}]))=0
	\]
	we must have $B=0$ (see \eqref{eqn:b1y1}).
\end{proof}

Next, we discuss a decomposition of $S^2(\Real^m)\otimes S^2_0(\Real^m)$. We start by defining a few subspaces of it.

Define
\begin{eqnarray*}
	Z_1&=& \set{ \Xi_1(B)|\, B\in S^2_0(\Real^m)}\\
	Z_2&=& \set{ \Xi_2(B)|\, B\in S^2_0(\Real^m)}\\
	Z_3&=& \set{\Xi_3(B)|\, B\in \Lambda^2(\Real^m)}\\
	Z_4&=& \set{ \Xi_4(c)|\, c\in \Real} 
\end{eqnarray*}
where
\begin{eqnarray*}
	\Xi_1(B)&=& [\delta_{(12)}B_{(34)}] \\
	\Xi_2(B)&=& [B_{(23)}\delta_{(14)}+ B_{(13)}\delta_{(24)}+B_{(24)}\delta_{(13)}+B_{(14)}\delta_{(23)}-\frac{4}{m} (B_{(34)}\delta_{(12)}+B_{(12)}\delta_{(34)})]\\
	\Xi_3(B)&=& [B_{(23)}\delta_{(14)}+ B_{(13)}\delta_{(24)}+B_{(24)}\delta_{(13)}+B_{(14)}\delta_{(23)}]\\
	\Xi_4(c)&=& c[\delta_{(13)}\delta_{(24)}+\delta_{(14)}\delta_{(23)}-\frac{2}{m}\delta_{(12)}\delta_{(34)}].
\end{eqnarray*}
We also define $H_{2,2}$ to be the set of all harmonic tensors in $S^2(\Real^m)\otimes S^2(\Real^m)$.

\begin{lem}
	\label{lem:fulldecom}
	We have the following decomposition
	\[
		S^2(\Real^m)\otimes S^2_0(\Real^m) = Z_1\oplus Z_2 \oplus Z_3 \oplus Z_4 \oplus H_{2,2}.
	\]
\end{lem}
\begin{proof}
	One could verify by their definitions that these are indeed subspaces in $S^2(\Real^m)\otimes S^2_0(\Real^m)$. Obviously,
	\[
		S^2(\Real^m)\otimes S^2_0(\Real^m) =Z_1 \oplus (S^2_0(\Real^m)\otimes S^2_0(\Real^m)).
	\]
	For tensors in $S^2_0(\Real^m)\otimes S^2_0(\Real^m)$, the only nontrivial trace is $T_{(23)}$. We check
	\begin{eqnarray*}
		T_{(23)} \Xi_2(B)&=& (m+2-\frac{8}{m})B \qquad \text{} \quad \forall B\in S^2_0(\Real^m) \\
		T_{(23)} \Xi_3(B)&=& (m+2)B \qquad \text{} \quad \forall B\in \Lambda^2(\Real^m)
	\end{eqnarray*}
	and
	\[
		T_{(23)}( [\delta_{(13)}\delta_{(24)}+\delta_{(14)}\delta_{(23)}-\frac{2}{m}\delta_{(12)}\delta_{(34)}]) = (m+1-\frac{2}{m})[\delta_{(12)}].
	\]
	We derive from these equations that $Z_2$, $Z_3$ and $Z_4$ have no nontrivial intersections. For any $A\in S_0^2(\Real^m)\otimes S_0^2(\Real^m)$, $T_{(23)}(A)$ is in $\Real^m\otimes \Real^m$ and it is uniquely a sum of tensors in $S_0^2(\Real^m)$, $\Lambda^2(\Real^m)$ and $[\delta_{(12)}] \Real$, from which we find a tensor in $Z_2\oplus Z_3\oplus Z_4$ such that by subtracting it from $A$ we obtain a harmonic tensor in $H_{2,2}$. This completes the proof.
\end{proof}

\begin{lem}
	\label{lem:computeB2}
	\[
		(Z_1\oplus Z_2\oplus Z_3\oplus Z_4) \bigcap {\rm ker}({\mathbf B}_2)= Z_3\oplus Z_4\oplus Z_5 
	\]
	where
	\[
		Z_5= \set{\Xi_5(B)|\, B\in S_0^2(\Real^m)}
	\]
	and
	\[
		\Xi_5(B)= [B_{(23)}\delta_{(14)}+ B_{(13)}\delta_{(24)}+B_{(24)}\delta_{(13)}+B_{(14)}\delta_{(23)}-\frac{4}{m} B_{(12)}\delta_{(34)}].
	\]
\end{lem}
\begin{proof}
	We begin with some results that can be proved by computations.
	\begin{eqnarray}
		\label{eqn:b21}
		{\mathbf B}_2( \Xi_1(B))&=& S_{(234)}\left( (4-2m) B_{(12)}\delta_{(34)} + (m^2-4)B_{(34)}\delta_{(12)} \right) \\
		\label{eqn:b22}
		{\mathbf B}_2(\Xi_2(B)) &=& S_{(234)} [\frac{8(m-2)}{m}B_{(12)}\delta_{(34)} + \frac{4(4-m^2)}{m}B_{(34)}\delta_{(12)}] \\
		\label{eqn:b23}
		{\mathbf B}_2(\Xi_3(B)) &=& 0 \\
		\label{eqn:b24}
		{\mathbf B}_2(\Xi_4(1)) &=& 0.
	\end{eqnarray}

	For \eqref{eqn:b21},
	\begin{eqnarray*}
		{\mathbf B}_2( [\delta_{(12)}B_{(34)}])&=& S_{(234)}\left( (4-2m) B_{(14)}\delta_{(23)} + (m^2-4)B_{(24)}\delta_{(13)} \right)\\
						       &=& S_{(234)}\left( (4-2m) B_{(12)}\delta_{(34)} + (m^2-4)B_{(34)}\delta_{(12)} \right).
	\end{eqnarray*}

	For \eqref{eqn:b22}, we notice that $T_{(12)}(A)=0$ for any $A\in Z_2$ so that the definition for ${\mathbf B}_2$ is simplified.
	\begin{eqnarray*}
		{\mathbf B}_2( \Xi_2(B)) &=& S_{(234)} (4 \Xi_2(B)_{(a1a4)}\delta_{(23)} - 2(m+2) \Xi_2(B)_{(3124)}) \\
					 &=& S_{(234)}\left[ 4(m+2-\frac{8}{m})B_{(14)}\delta_{(23)} -2(m+2)( B_{(12)}\delta_{(34)}+B_{(32)}\delta_{(14)}\right. \\
					 && \left. +B_{(14)}\delta_{(32)}+B_{(34)}\delta_{(12)} -\frac{4}{m}(B_{(24)}\delta_{(13)}+B_{(31)}\delta_{(24)})) \right].
	\end{eqnarray*}
	Since we will do symmetrization of $(234)$ anyway, among all terms above, there are only two distinct ones, $B_{(12)}\delta_{(34)}$ and $B_{(34)}\delta_{(12)}$. Hence,
	\[
		{\mathbf B}_2(\Xi_2(B)) = S_{(234)} [\frac{8(m-2)}{m}B_{(12)}\delta_{(34)} + \frac{4(4-m^2)}{m}B_{(34)}\delta_{(12)}].
	\]

	For \eqref{eqn:b23}, 
	\begin{eqnarray*}
		{\mathbf B}_2(\Xi_3(B)) &=&S_{(234)} (4 \Xi_3(B)_{(a1a4)}\delta_{(23)} - 2(m+2) \Xi_3(B)_{(3124)}) \\
					&=& S_{(234)} \left( (2m+4)B_{(14)}\delta_{(32)} -(2m+4)(B_{(34)}\delta_{(12)}+B_{(12)}\delta_{(34)}+B_{(32)}\delta_{(14)}) \right) \\
					&=&S_{(234)}\left( -4(m+2) B_{(34)}\delta_{(12)} \right) =0.
	\end{eqnarray*}

	For \eqref{eqn:b24},
	\begin{eqnarray*}
	&& {\mathbf B}_2(\Xi_4(1))\\
	&=&S_{(234)} \left( (2m-\frac{8}{m}) \delta_{(14)}\delta_{(23)}-2(m+2)\delta_{(12)}\delta_{(34)}+ \frac{4(m+2)}{m}\delta_{(13)}\delta_{(24)} \right) \\
	&=&0.
	\end{eqnarray*}

	It follows from \eqref{eqn:b21} and \eqref{eqn:b22} that
	$$
	{\mathbf B}_2(\Xi_5(B))=0
	$$
	because our definition of $\Xi_5$ is equivalent to
	$$
	\Xi_5(B)= \Xi_2(B) + \frac{4}{m} \Xi_1(B), \qquad \text{} \quad \forall B\in S_0^2(\Real^m).
	$$
	Together with \eqref{eqn:b23} and \eqref{eqn:b24}, we find that $Z_3\oplus Z_4\oplus Z_5$ is in ${\rm Ker}({\mathbf B}_2)$. 

	To finish the proof, we claim that for $B, B'\in S^2_0(\Real^m)$ satisfying
	\begin{equation}
		\label{eqn:claim}
		{\mathbf B}_2 \left( \Xi_1(B)+ \Xi_2(B') \right)=0,
	\end{equation}
	we have
	$$
	\Xi_1(B)+\Xi_2(B')\in Z_5.
	$$

	Taking $T_{(12)}$ of \eqref{eqn:b21}, we obtain
	\begin{eqnarray*}
		T_{(12)} {\mathbf B}_2(\Xi_1(B)) &=& \frac{1}{3}T_{(12)} \left( (4-2m)B_{(12)}\delta_{(34)} + (m^2-4) B_{(34)}\delta_{(12)} \right) \\
						 &&+  \frac{1}{3}T_{(12)} \left( (4-2m)B_{(13)}\delta_{(24)} + (m^2-4) B_{(24)}\delta_{(13)} \right) \\
						 &&+  \frac{1}{3}T_{(12)} \left( (4-2m)B_{(14)}\delta_{(23)} + (m^2-4) B_{(23)}\delta_{(14)} \right) \\
						 &=& \frac{1}{3}\left( (m^2-4)m+ (4-2m)*2 + (m^2-4)*2 \right) B_{(12)}\\
						 &=&\frac{1}{3}m(m+4)(m-2) B_{(12)},
	\end{eqnarray*}
	and from \eqref{eqn:b22}
	$$
	T_{(12)} {\mathbf B}_2(\Xi_2(B')) = -\frac{4}{3}(m+4)(m-2) B'_{(12)}.
	$$
	Hence, by taking trace $T_{(12)}$ of \eqref{eqn:claim}, we find that
	$$
	mB-4B'=0,
	$$
	which finishes our proof of the claim and the lemma.
\end{proof}

The Weyl tensor appears in the decomposition and the kernel of ${\mathbf B}_2$.
\begin{lem}
	\label{lem:weyl} 
	\[
		H_{2,2}\cap {\rm ker}({\mathbf B}_2) = \widetilde{\mathcal W}=s(\mathcal W).
	\]
\end{lem}
\begin{proof}
	By the definition of ${\mathbf B}_2$, we derive that for some $A\in H_{2,2}$, 
	\[
		{\mathbf B}_2(A)=0 \Longleftrightarrow A_{(1234)} +A_{(1342)}+A_{(1423)}=0.
	\]
	Hence, a tensor $A$ lies in $H_{2,2}\cap {\rm Ker}({\mathbf B}_2)$ if and only if
	\begin{enumerate}[(a)]
		\item $A_{(1234)}=A_{(2134)}$ and $A_{(1234)}=A_{(1243)}$;
		\item all traces vanish;
		\item $A_{(1234)} +A_{(1342)}+A_{(1423)}=0$.
	\end{enumerate}
	It is a corollary of (a-c) that 
	\begin{enumerate}[(d)]
		\item 	$A_{(1234)}=A_{(3412)}$.
	\end{enumerate}

	{\bf Step 1.} $H_{2,2}\cap {\rm Ker}{\mathbf B}_2\subset \widetilde{\mathcal W}$.

	By the definition of $\widetilde{\mathcal W}$, it amounts to proving for each $A\in H_{2,2}\cap {\rm Ker}({\bf B}_2)$, there is $W\in \mathcal W$ such that
	$$
	A= s(W).
	$$
	Given $A$, set
	$$
	W= [-\frac{2}{3}A_{(1234)}-\frac{4}{3}A_{(1324)}].
	$$
	It remains to check that $W$ is a Weyl tensor and $A=s(W)$.

	Recall that $W$ is a Weyl tensor if and only if 
	\begin{enumerate}[(i)]
		\item $W_{(1234)}=-W_{(2134)}$ and $W_{(1234)}=-W_{(1243)}$;
		\item all traces vanish;
		\item $W_{(1234)}+ W_{(1342)}+W_{(1423)}=0$.
	\end{enumerate}
	We derive from (i)-(iii) that $W_{(1234)}=W_{(3412)}$.

	Obviously, (ii) and (iii) for $W$ follow from (b) and (c) of $A$. For (i), we compute
	\begin{eqnarray*}
		&& W_{(1234)}+W_{(2134)}\\
		&=& -\left( \frac{2}{3}A_{(1234)}+\frac{4}{3}A_{(1324)}+ \frac{2}{3}A_{(2134)}+\frac{4}{3}A_{(2314)}\right) \\
		&=& - \frac{4}{3} \left( A_{(1234)} + A_{(1324)}+A_{(1423)}\right) \\
		&=&0.
	\end{eqnarray*}
	Here in the third line above, we used $A_{(2314)}=A_{(1423)}$ and in the last line above, we used (c). The other half is proved similarly. Hence, we have verified that $W\in \mathcal W$.

	It remains to check that
	\[
		s(W)=A.
	\]
	In fact, direct computation shows
	\begin{eqnarray*}
		s(W)_{(1234)}&=& s([-\frac{2}{3}A_{(1234)}-\frac{4}{3}A_{(1324)}])_{(1234)} \\
			     &=&\frac{1}{2}\left( [-\frac{2}{3}A_{(1234)}-\frac{4}{3}A_{(1324)}]_{(1342)}+ [-\frac{2}{3}A_{(1234)}-\frac{4}{3}A_{(1324)}]_{(1432)} \right) \\
			     &=&\frac{1}{2} \left( -\frac{2}{3}A_{(1342)}-\frac{4}{3}A_{(1423)} -\frac{2}{3}A_{(1432)}-\frac{4}{3}A_{(1324)} \right)\\
			     &=& -A_{(1324)}-A_{(1423)}=A_{(1234)}.
	\end{eqnarray*}

	{\bf Step 2.} $H_{2,2}\cap {\rm Ker}{\mathbf B}_2\supset \widetilde{\mathcal W}$. Namely, for each $W\in \mathcal W$, we need to show
	$$
	s(W) \in H_{2,2}\cap{\rm Ker}{\mathbf B}_2. 
	$$
	By the definition of $s(W)$ 
	\[
		s(W)=\frac{1}{2} ( W_{(1342)}+W_{(1432)}),
	\]
	one verifies (a)-(c) directly.
\end{proof}

Combining Lemma \ref{lem:computeB2} and Lemma \ref{lem:weyl}, we obtain
\begin{prop}
	\label{prop:decomp2}
	\[
		S^2(\Real^m)\otimes S^2_0(\Real^m) \bigcap {\rm ker}({\mathbf B}_2) = Z_3\oplus Z_4\oplus Z_5\oplus \widetilde{\mathcal W}.
	\]
	Moreover, using the definition of $\Psi_4$, 
	\begin{eqnarray*}
		Z_3 &=& \Psi_4(\Lambda^2(\Real^m)) \\
		Z_4 &=& \Psi_4( \set{c [\delta_{(12)}]| \, c\in \mathbb R})\\
		Z_5 &=&\Psi_4(S_0^2(\Real^m)).
	\end{eqnarray*}
\end{prop}

\begin{proof}
	The moreover part of the proposition is a direct consequence of the definitions of $\Psi_4$ and $Z_3, \cdots , Z_5$. By Lemma \ref{lem:fulldecom}, Lemma \ref{lem:computeB2} and Lemma \ref{lem:weyl}, 
	\[
		S^2(\Real^m)\otimes S^2_0(\Real^m) \bigcap {\rm ker}({\mathbf B}_2) \supset Z_3\oplus {Z_4}\oplus Z_5\oplus \widetilde{\mathcal W}.
	\]
	It remains to prove the reverse direction. For any $A$ in the left-hand side, by the definition of ${\mathbf B}_2$, we have
		\[
			S_{(234)}[4A_{(a1a4)}\delta_{(23)}-2(m+2)A_{(3124)}-2A_{(aa14)}\delta_{(23)}+(m+2)A_{(aa24)}\delta_{(13)}]=0.
		\]
	By taking $T_{(12)}$ of the above equation and noticing that $T_{(34)}A=0$, we have the following equation of $2$-tensors:
	$$
	(-4m) T_{(23)}A + (m^2+2m-4) T_{(12)}A + \operatorname{Tr}( T_{(23)}A) \delta =0.
	$$
	By taking the traceless symmetric part of it and noticing that $T_{(12)}A \in S^2_0(\mathbb R^m)$, we get
	$$
	(-4m) \operatorname{Sym}_0 T_{(23)} A + (m^2+2m -4) T_{(12)}A =0.
	$$
	The key fact in this proof is that the same equation holds for $\Xi_5(B)$ for any $B\in S^2_0(\mathbb R^m)$, i.e.
	$$
	(-4m) \operatorname{Sym}_0 T_{(23)} \Xi_5(B) + (m^2+2m -4) T_{(12)} \Xi_5(B) =0.
	$$
	This can be verified directly by the definition of $\Xi_5$. We also note that $T_{(12)}$ is an isomorphism between $Z_5$ and $S^2_0(\mathbb R^m)$. Therefore, for $A$ as above, there is $A_5\in Z_5$ such that $T_{(12)}(A-A_5)$ and $\operatorname{Sym}_0(T_{(23)})(A-A_5)$ vanish at the same time. As in the proof of Lemma \ref{lem:fulldecom}, by further subtracting $A_4\in Z_4$ and $A_3\in Z_3$, we are left with a tensor in ${\rm Ker}({\mathbf B}_2)$ that is totally traceless. The proof of this proposition is done by using Lemma \ref{lem:weyl} again.
\end{proof}

\section{Bianchi coordinate at infinity}
\label{sec:analysis}

The aim of this section is to show how we can turn an almost Bianchi coordinate into a Bianchi one. This amounts to solving a harmonic map equation using the Schauder fixed point theorem.

We assume that a natural coordinate system is given for $\Omega_R:=\Real^m \setminus B_R\subset \Real^m$. Let $g_e$ be the flat metric given by this coordinate. A map $\varphi$ from $\Omega_R$ to $\Omega_R$ is a vector-valued function. It is harmonic from $(\Omega_R,g_e)$ to $(\Omega_R,g)$ when
\[
	\triangle \varphi_i + \Gamma^i_{jk}(\varphi(x)) \pfrac{\varphi_j}{x_a} \pfrac{\varphi_k}{x_a}=0, \qquad \text{for} \quad  i=1,\cdots ,m
\]
Here $\Gamma$ is the Christoffel symbol of the metric $g$.

\begin{thm}
	\label{thm:pde} Given a Riemannian metric $g$ defined in $\Omega_R$ satisfying
	\begin{equation}
		\label{eqn:almostflat}
		g_{ij}-\delta_{ij}\in O(\abs{x}^{-(\tilde{n}-2)})
	\end{equation}
	for some $\tilde{n}>2$, assume the identity map from $\Omega_R$ to itself is approximately a harmonic map from $(\Omega_R,g_e)$ to $(\Omega_R,g)$ in the sense that
	\begin{equation}
		\label{eqn:hmapprox}
		\Gamma^i_{aa}(x) = O(\abs{x}^{-\tilde{n}}).
	\end{equation}
	Then for any $\varepsilon>0$, there is some $R'$ sufficiently large and an $\Real^m$-valued function $u$ defined on $\Omega_{R'}$ such that $u$ solves 
	\begin{equation}
		\label{eqn:hmpde}
		\triangle u_i + \Gamma^i_{jk}(x+u(x)) \left( \delta_{ja}+\pfrac{u_j}{x_a} \right)\left( \delta_{ka}+\pfrac{u_k}{x_a} \right)=0
	\end{equation}
	and
	\[
		u\in O(\abs{x}^{-(\tilde{n}-2-\varepsilon)}).
	\]
	In particular, $\varphi(x)=x+u(x)$ is a harmonic map from $(\Omega_{R'},g_e)$ to $(\Omega_{R},g)$.
\end{thm}
The proof requires some estimates for the Poisson equation in weighted Schauder spaces, which will be discussed in Section \ref{sub:weighted} and Section \ref{sub:poisson}. After that, we prove Theorem \ref{thm:pde}.

\subsection{Weighted function space}
\label{sub:weighted}
For $R>0$, recall that $\Omega_R= \Real^m \setminus B_R$ and set
\[
	S_r: B_2 \setminus B_1\to B_{2r}\setminus B_{r}
\]
to be the map
\[
	S_r(x)=r x.
\]
For functions defined on $\Omega_R$, we define the weighted H\"older norm ($\beta>0$)
\[
	\norm{v}_{\mathcal X_{\alpha,\beta}(\Omega_R)}= \sup_{r\geq R}  r^{\beta}\norm{v\circ S_r}_{C^\alpha(B_2 \setminus B_1)}.
\]
We denote the corresponding Banach space by $\mathcal X_{\alpha,\beta}(\Omega_R)$. When there is no ambiguity, we often omit $\Omega_R$ and write $\mathcal X_{\alpha,\beta}$ only.

\begin{rem}
	Here in the above definitions, we allow $\alpha$ to be any positive real number other than integers. For example, by $C^{5/2}$, we mean the set of functions whose derivatives up to second order are H\"older continuous with exponent $1/2$.
\end{rem}

\begin{rem}
	\label{rem:kalf}
	The exact value of $\alpha$ is not important for the proofs in this paper. Later, for technical convenience, we may require $\alpha$ to be bigger than some dimensional constant so that we can use the fact that: a $C^\alpha$ function on $S^{m-1}$ has its spherical harmonic expansion converge uniformly. See \cite{kalf1995}.
\end{rem}

We discuss a few properties of the weighted functions.

For the derivative, we have
\begin{equation}
	\label{eqn:partial}
	\norm{\partial_{x_i} v}_{\mathcal X_{\alpha-1,\beta+1}(\Omega_R)}\leq C \norm{v}_{\mathcal X_{\alpha,\beta}(\Omega_R)} \qquad \text{} \quad \forall i=1,\cdots m.
\end{equation}
To see this, for any $r\geq R$,
\begin{eqnarray*}
	r^{\beta+1} \norm{(\partial_{x_i} v)\circ S_r}_{C^{\alpha-1}(B_2\setminus B_1)} &=& r^{\beta+1} \norm{ r^{-1} \partial_{x_i}(v\circ S_r)}_{C^{\alpha-1}(B_2\setminus B_1)} \\
											&\leq& r^{\beta} \norm{ v\circ S_r}_{C^\alpha(B_2\setminus B_1)}.
\end{eqnarray*}

For the multiplication, we have
\begin{equation}
	\label{eqn:multi}
	\norm{fg}_{\mathcal X_{\min(\alpha_1,\alpha_2),\beta_1+\beta_2}(\Omega_R)} \leq C \norm{f}_{\mathcal X_{\alpha_1,\beta_1}(\Omega_R)} \norm{g}_{\mathcal X_{\alpha_2,\beta_2}(\Omega_R)}.
\end{equation}
The proof is omitted.

We will also need the following result about the composition of weighted functions.
\begin{lem}
	\label{lem:composit}
	Assume that $f=O(\abs{x}^{-l})$ for some $l>0$ on $\Omega_{R/2}$. For $\Real^m$-valued function $u$ on $\Omega_R$, set
	\[
		\varphi(x)=x+u(x).
	\]
	If for some $\beta>0$
	\[
		\norm{u}_{\mathcal X_{\alpha,\beta}(\Omega_R)}\leq 1,
	\]
	then as long as $R$ is larger than some universal number, we have
	\[
		\norm{f\circ \varphi}_{\mathcal X_{\alpha,l}(\Omega_{R})}\leq C.
	\]
\end{lem}
\begin{proof}
	For $r\geq R$, we need to bound
	\[
		r^l \norm{f\circ \varphi \circ S_r}_{C^\alpha(B_2\setminus B_1)}.
	\]
	By the assumption of $f$,
	\[
		r^l \norm{f\circ S_r}_{C^\alpha(B_2\setminus B_1)}\leq C \qquad \text{for all} \quad r>R/2.
	\]
	Since $f\circ \varphi\circ S_r = f\circ S_r \circ S_r^{-1}\circ \varphi \circ S_r$, it remains to show
	\begin{equation}
		\label{eqn:srphi}
		S_r^{-1}\circ \varphi\circ S_r: B_2\setminus B_1 \to B_{4}\setminus B_{1/2}
	\end{equation}
	is bounded in $C^\alpha$. In fact, we have
	\[
		S_r^{-1}\circ \varphi \circ S_r (x)= x+ r^{-1} u(r x)
	\]
	and \eqref{eqn:srphi} follows from our assumption for $u$, in particular,
	\[
		\norm{u\circ S_r}_{C^\alpha(B_2\setminus B_1)}\leq 1.
	\]
\end{proof}

\subsection{Poisson equation on $\Omega_R$}
\label{sub:poisson}

The aim of this subsection is to prove
\begin{thm}
	\label{thm:linear} 
	There is a bounded linear map
	\[
		\Psi: \mathcal X_{\alpha-2,\beta+2}(\Omega_R) \to \mathcal X_{\alpha,\beta}(\Omega_R)
	\]
	such that
	\[
		\triangle \Psi(f) = f \qquad \text{on} \quad \Omega_R
	\]
	for any $f\in \mathcal X_{\alpha-2,\beta+2}(\Omega_R)$. Moreover, there is a constant $C$ independent of $R$ such that
	\[
		\norm{\Psi(f)}_{\mathcal X_{\alpha,\beta}(\Omega_R)}\leq C \norm{f}_{\mathcal X_{\alpha-2,\beta+2}(\Omega_R)}.
	\]
\end{thm}

\begin{proof}[Proof of Theorem \ref{thm:linear}]
	This is a standard weighted estimate for the Poisson equation on an exterior domain. Since $\beta\notin\mathbb Z$, write $\beta=\rho+\varepsilon$, where $\rho\in\mathbb Z_{\geq 0}$ and $\varepsilon\in(0,1)$. Lemma 5 of \cite{meyers1963} gives a linear particular solution with decay $O(\abs{x}^{-\beta})$, while Theorem 1 therein, together with the usual interior Schauder estimates on dyadic annuli, gives
	\[
		\norm{u}_{\mathcal X_{\alpha,\beta}(\Omega_R)}
		\leq C
		\norm{f}_{\mathcal X_{\alpha-2,\beta+2}(\Omega_R)}.
	\]
	The constant is independent of $R$ by dilation.
\end{proof}

\subsection{Proof of Theorem \ref{thm:pde}}

Since $\tilde{n}$ is assumed to be larger than $2$, we may assume without loss of generality that $\varepsilon$ is very small so that we may set $\beta=\tilde{n}-\varepsilon>2$ and choose $\beta'\in (\beta,\tilde{n})$. Assume that $\beta$ and $\beta'$ are not integers. 
For simplicity of notation, we set
\[
	F[u]_i:= \Gamma^i_{jk}(x+u(x)) \left( \delta_{ja}+\pfrac{u_j}{x_a} \right)\left( \delta_{ka}+\pfrac{u_k}{x_a} \right).
\]

\begin{lem}\label{lem:Fu}
	For any $\delta>0$, there is $R$ sufficiently large (depending on the metric and $\delta$, $\beta'$, $\tilde{n}$) so that for any $u\in \mathcal X_{\alpha,\beta-2}(\Omega_R)$ with $\norm{u}_{\mathcal X_{\alpha,\beta-2}(\Omega_R)}\leq 1$, we have
	\[
		\norm{F[u]}_{\mathcal X_{\alpha-1,\beta'}(\Omega_{R})}\leq \delta.
	\]
\end{lem}
\begin{proof}
	The proof consists of two steps.

	{\bf Step 1.}
	There exists some constant $C$ such that
	\begin{equation}
		\label{eqn:step1}
		\norm{F[u]}_{\mathcal X_{\alpha-1,\tilde{n}}(\Omega_{R})}\leq C.
	\end{equation}

	For the proof, we compute
	\[
		F[u]_i = \Gamma^i_{aa}(x+u(x)) + 2 \Gamma^i_{ja}(x+u(x)) \pfrac{u_j}{x_a}  + \Gamma^i_{jk}(x+u(x)) \pfrac{u_j}{x_a} \pfrac{u_k}{x_a}.
	\]
	Given the assumptions of Theorem \ref{thm:pde}, we have
	\[
		\Gamma^i_{jk}(x) = O(\abs{x}^{-(\tilde{n}-1)}), \qquad \Gamma^i_{aa}(x) = O(\abs{x}^{-\tilde{n}}).
	\]
	Lemma \ref{lem:composit} implies that
	\[
		\norm{\Gamma^i_{jk}(x+u(x)) }_{\mathcal X_{\alpha,\tilde{n}-1}(\Omega_R)}\leq C
	\]
	and
	\[
		\norm{\Gamma^i_{aa}(x+u(x))}_{\mathcal X_{\alpha,\tilde{n}}(\Omega_R)}\leq C.
	\]
	Using \eqref{eqn:partial} and the assumption of $u$, we have
	\[
		\norm{\pfrac{u_j}{x_a}}_{\mathcal X_{\alpha-1,\beta-1}(\Omega_R)}\leq C.
	\]
	Noting that $\beta-1>1$, we can prove \eqref{eqn:step1} by using the multiplication rule in \eqref{eqn:multi}.

	{\bf Step 2.}
	With \eqref{eqn:step1}, the proof of the lemma is done by taking $R$ large (depending on $\delta$ and the constant $C$ in \eqref{eqn:step1}) in the following inequality
	\begin{equation*}
		\norm{F[u]}_{\mathcal X_{\alpha-1,\beta'}(\Omega_{R})}\leq \norm{F[u]}_{\mathcal X_{\alpha-1,\tilde{n}}(\Omega_{R})} (R)^{(\beta'-\tilde{n})}.
	\end{equation*}
	The above inequality is a direct consequence of the definition of the weighted norm. In fact, for any function $w\in \mathcal X_{\alpha-1,\tilde{n}}(\Omega_{R})$, we have
	\begin{eqnarray*}
		&& \norm{w}_{\mathcal X_{\alpha-1,\beta'}(\Omega_{R})} \\
		&=&\sup_{r\geq R} r^{\beta'} \norm{w\circ S_r}_{C^{\alpha-1}(B_2\setminus B_1)} \\
		&\leq& (R)^{\beta'-\tilde{n}}\sup_{r\geq R} r^{\tilde{n}} \norm{w\circ S_r}_{C^{\alpha-1}(B_2\setminus B_1)} \\
		&=& (R)^{\beta'-\tilde{n}} \norm{w}_{\mathcal X_{\alpha-1,\tilde{n}}(\Omega_R)}.
	\end{eqnarray*}
\end{proof}

We now turn to the proof of Theorem \ref{thm:pde}. 

For the constant $C$ in Theorem \ref{thm:linear}, choose $\delta=\frac{1}{C}$. Let $R'$ be the $R$ given by Lemma \ref{lem:Fu}. Consider the unit ball $\mathcal B$ in the Banach space $\mathcal X_{\alpha,\beta-2}(\Omega_{R'})$. For each $u\in \mathcal B$, Lemma \ref{lem:Fu} implies that
\[
	\norm{F[u]}_{\mathcal X_{\alpha-1,\beta'}(\Omega_{R'})}\leq \delta.
\]
Theorem \ref{thm:linear} gives some $v\in \mathcal X_{\alpha+1,\beta'-2}(\Omega_{R'})$ satisfying
\[
	\norm{v}_{\mathcal X_{\alpha+1,\beta'-2}(\Omega_{R'})}\leq 1.
\]
Consider the natural embedding
\[
	\iota: \mathcal X_{\alpha+1,\beta'-2}(\Omega_{R'})\to \mathcal X_{\alpha,\beta-2}(\Omega_{R'}).
\]
For $\beta'>\beta$, $\iota$ is a compact linear operator satisfying
\[
	\norm{\iota(w)}_{\mathcal X_{\alpha,\beta-2}(\Omega_{R'})}\leq \norm{w}_{\mathcal X_{\alpha+1,\beta'-2}(\Omega_{R'})}.
\]
By defining
\[
	u\mapsto - \iota\circ \Psi (F[u]),
\]
we obtain a compact operator from $\mathcal B$ to itself. Then the Schauder fixed point theorem implies the existence of $u$ in $\mathcal B$ with
\[
	\triangle u = - F[u].
\]

\section{Expansion of metrics at a smooth point}\label{sec:smooth}

In this section, we study special coordinates around a smooth point of any smooth Riemannian manifold. We {\it do not claim} any result of this section to be new. Indeed, the first part of Theorem \ref{thm:known} is well documented in any Riemannian geometry textbooks and the proof we used here is nothing but the one outlined by Bryant in \cite{web} with the algebraic details worked out. The second claim in Theorem \ref{thm:known} is rather routine in PDE. It shows that we can have a harmonic/Bianchi coordinate that is as good as the normal coordinate in the sense that we have the same expansion of the metric.

\subsection{Definitions and results}

Given a smooth Riemannian manifold, around a fixed point, {\it the geodesic coordinate} is defined via exponential map and it has many advantages, in particular, the metric tensor has an expansion in this coordinate where the Riemannian curvature tensor appears as the coefficients of the second order term. 
\begin{equation}
	\label{eqn:exp_geodesic}
	g_{ij}= \delta_{ij}- \frac{1}{3}R_{ipqj}x_p x_q + \cdots .
\end{equation}
Besides this, one can define a {\it harmonic coordinate} in a neighborhood, where the coordinates are harmonic functions with respect to the metric. The advantage of this harmonic coordinate is that in this coordinate, the leading term of Ricci curvature is nothing but the Laplacian of the metric tensor (see for example \cite{deturk})
\[
	R_{ij}= -\frac{1}{2}\triangle g_{ij} + \cdots 
\]
It is natural to ask if there is any coordinate that has both the features mentioned above. 

There is also a notion of {\it Bianchi coordinate}, which is defined by
\begin{equation}
	\label{eqn:bianchicoordinate}
	B_{g_e} g = 0
\end{equation}
where $g_e$ is the flat metric defined by the coordinates and $B_g(h)=\delta(h-\frac{1}{2}({\rm Tr}_g h) g)$ for any symmetric $2$-tensor $h$. This coordinate serves a similar purpose as harmonic coordinate in making the Ricci equation elliptic, while its definition allows global considerations. It is widely used in the study of deformation of Einstein metrics.
We have the following interpretation of the Bianchi coordinate
\begin{lem}
	\label{lem:bianchiharmonic}
	Assume that $(M,g)$ is an $m$-dimensional Riemannian manifold and $\bx \in M$, $U$ is a neighborhood on which $\varphi:B\subset \Real^m\to U$ defines a coordinate system. Then it is Bianchi if and only if $\varphi$ is a harmonic map from $B$ (with the flat metric of $\Real^m$) to $(U,g)$.
\end{lem}
\begin{proof}
	If we write $\tilde{g}$ for $\varphi^*(g)$ as a metric on $B$, the Bianchi coordinate assumption means
	\begin{equation}
		\label{eqn:bianchimeans}
		- \partial_j \tilde{g}_{ij} + \frac{1}{2} \partial_i (\sum_j \tilde{g}_{jj})=0.
	\end{equation}
	Let $\Gamma^i_{jk}$ be the Christoffel symbols of the Levi-Civita connection of $\tilde{g}$. 
	For the identity map to be a harmonic map from $(B,g_e)$ to $(B,g)$, we need
	\begin{equation}
		\label{eqn:beingharmonic}
		\triangle_{g_e} x^i + \Gamma^i_{jk} \pfrac{x_j}{x_l}\pfrac{x_k}{x_l}=0,
	\end{equation}
	which is equivalent to 
	\[
		\Gamma^i_{jj}=0.
	\]
	By the definition of $\Gamma$,
	\[
		\Gamma^i_{jj}= g^{il}\left( \pfrac{g_{jl}}{x_j} - \frac{1}{2}\pfrac{g_{jj}}{x_l} \right).
	\]
	Since $g^{il}$ is invertible, we find that \eqref{eqn:beingharmonic} is equivalent to \eqref{eqn:bianchimeans}.
\end{proof}
\begin{rem}
	There is another way of looking at Lemma \ref{lem:bianchiharmonic}.
	Since the map $\varphi$ is a diffeomorphism, then a variation of $\varphi$ is equivalent (by the invariance of Dirichlet energy under the action of diffeomorphism group) to a variation from the domain. Hence, in the terminology of harmonic maps, for such maps, being harmonic is equivalent to being stationary harmonic. The divergence free property of the Stress-energy tensor in this case is exactly the same as the Bianchi condition \eqref{eqn:bianchicoordinate}. 
\end{rem}

The main result of this section is
\begin{thm}\label{thm:known}
	Let $(M,g)$ be any Riemannian manifold of dimension $m$ and $\bx \in M$ be any fixed point. Then there is a coordinate around $\bx$ such that
	\[
		g_{ij}= \delta_{ij}- \frac{1}{3}R_{ipqj}x_p x_q + o(\abs{x}^2),
	\]
	where $R_{ipqj}$ is the Riemannian curvature tensor at $\bx$.
	Moreover, if $Ric(\bx)=0$, the above coordinate can be chosen so that it is harmonic or Bianchi.
\end{thm}

\subsection{Proof of Theorem \ref{thm:known}}

For a vector-valued function $u: B(0,\varepsilon)\subset \Real^m\to \Real^m$, consider the equation
\begin{equation}
	\label{eqn:pde}
	\triangle u = F(x,u,\nabla u)
\end{equation}
where $F$ is a smooth function in its arguments.  We are interested in the ``local solvability", which is the topic of Section 14.3 of \cite{taylorbook}. Note that the author discussed more general elliptic equation of order $m$.

The $k=0$ case of the following result is exactly Proposition 3.3 (for our equation) in \cite{taylorbook}, which was presented as a ``refinement" of the local solvability theorem (Theorem 3.1 therein). As an interesting historic note, the author remarked after Proposition 3.3: ``The desirability of having this refinement was pointed out to the author by R. Bryant."
\begin{lem}
	\label{lem:pde} If $u_0$ is a smooth function on $B(0,\varepsilon)$ that satisfies \eqref{eqn:pde} at $0$ up to order $k$, namely,
	\[
		\left( \triangle u_0 - F(x,u_0,\nabla u_0) \right)^{(l)}(0)=0
	\]
	for $l=0,1,\cdots ,k$. Then there is a solution $u$ to \eqref{eqn:pde} defined on some smaller neighborhood $B(0,\delta)$ such that
	\[
		(u-u_0)^{(l)}(0)=0
	\]
	for any $l=0,1,\cdots ,k+2$.
\end{lem}
The generalization from $k=0$ to general $k$ is routine. Since the initial approximate solution is better, the formal solution in Lemma 3.2 could be chosen better. Namely, $m+1$ in (3.13) could be improved to $m+1+k$. The rest of the argument could be modified accordingly. We refer the details of this proof to Section 14.3 of \cite{taylorbook}. We also note that a similar version of the weighted function space and an application of the Schauder fixed point theorem as in Section \ref{sec:analysis} could also provide a proof of Lemma \ref{lem:pde}.

Now, let's start the proof of Theorem \ref{thm:known}. 

Find any local coordinate around $\bx\in M$. Since the metric is smooth, there is the Taylor's expansion. By making a linear transformation, we may assume that
\[
	g_{ij}(x)= \delta_{ij}+ O(\abs{x}).
\]
By doing so, we have simplified the constant term in the expansion. Lemma \ref{lem:prealgebra} deals with the linear term.
\begin{lem}
	\label{lem:prealgebra} There is another smooth coordinate centered at $\bx$ such that
	\begin{equation}
		\label{eqn:prealgebra}
		g_{ij}= \delta_{ij}+ O(\abs{x}^2).
	\end{equation}
\end{lem}
\begin{proof}
	In the original coordinate,
	\begin{equation}
		\label{eqn:step1ass}
		g_{ij}=\delta_{ij}+A_{ijk}x_k + O(\abs{x}^2)
	\end{equation}
	where $A_{ijk}=A_{jik}$.
	By Proposition  \ref{prop:psi1}, we can find $B_{ijk}$ satisfying
	\begin{equation}
		\label{eqn:step11}
		A_{abk}+2B_{kab}+2B_{kba}=0, \qquad \text{} \quad \forall k,a,b=1,\cdots ,m.
	\end{equation}
	Consider a coordinate change
	\[
		\varphi(x)_k = x_k + B_{ijk}x_ix_j.
	\]
	The pullback metric is
	\[
		(\varphi^* g)_{ab}= \left( \delta_{ij}+A_{ijk}\varphi_k \right) \left( \delta_{ia}+2B_{kai}x_k \right)\left( \delta_{jb}+ 2B_{kbj}x_k\right)+  O(\abs{x}^2).
	\]
	Thanks to \eqref{eqn:step11}, we obtain
	\[
		(\varphi^* g)_{ab}= \delta_{ab} + O(\abs{x}^2).
	\]
\end{proof}

Lemma \ref{lem:algebra} deals with the second order terms in the expansion of $g_{ij}$.
\begin{lem}
	\label{lem:algebra}
	There is another smooth coordinate centered at $\bx$ such that
	\[
		g_{ij}=\delta_{ij}- \frac{1}{3}R_{ipqj} x_px_q + O(\abs{x}^3)
	\]
	where $R_{ipqj}$ is the Riemann curvature tensor at $x=0$. 
\end{lem}

\begin{proof}
	Since $g$ is smooth and satisfies \eqref{eqn:prealgebra}, we have an expansion
	\[
		g_{ij}=\delta_{ij}+A_{ijkl}x_kx_l +O(\abs{x}^3).
	\]
	Here $A_{ijkl}$ satisfies $A_{ijkl}=A_{jikl}=A_{ijlk}$.

	By Proposition \ref{prop:psi2}, there are $\tilde{R}\in \tilde{\mathcal C}$ and $B\in \Real^m\otimes S^3(\Real^m)$ such that
	\[
		A= \tilde{R}+ \Psi_2(B).
	\]
	Consider a coordinate change
	\[
		\varphi(x)_i= x_i + B_{ijkl}x_jx_kx_l.
	\]

	The pullback metric is
	\begin{eqnarray*}
		(\varphi^* g)_{pq}&=& \pfrac{\varphi_i}{x_p}\pfrac{\varphi_j}{x_q} \left( \delta_{ij}+A_{ijkl}\varphi_k\varphi_l + O(\abs{\varphi}^3) \right) \\
				  &=& \delta_{pq} + 3B_{qpbc}x_bx_c+ 3B_{pqbc}x_bx_c + A_{pqbc}x_bx_c + O(\abs{x}^3).
	\end{eqnarray*}
	By the definition of $\Psi_2$ in Section \ref{subsec:maps}, 
	\[
		(\varphi^* g)_{ij} = \delta_{ij} + \tilde{R}_{ijpq} x_px_q + O(\abs{x}^3).
	\]
	By the isomorphism between $\tilde{\mathcal C}$ and $\mathcal C$, we have $R\in \mathcal C$ such that
	\[
		(\varphi^* g)_{ij} = \delta_{ij} -\frac{1}{6} (R_{ipqj}+R_{iqpj})x_px_q + O(\abs{x}^3).
	\]
	Given this expansion, we verify by direct computation that $R$ is the curvature tensor of $\varphi^* g$ at $x=0$.
\end{proof}
This proves the first claim in Theorem \ref{thm:known}. 

For the second part, take the coordinate given by Lemma \ref{lem:algebra} and we plan to solve
\begin{equation}
	\label{eqn:pde3}
	\triangle u_i + \Gamma^i_{jk}(x+u(x)) \left( \delta_{ja} + \pfrac{u_j}{x_a} \right)\left( \delta_{ka}+\pfrac{u_k}{x_a} \right)=0
\end{equation}
for $u$ in a neighborhood of $0$ such that
\begin{equation}
	\label{eqn:zerou}
	u(0)= \nabla u(0) = \nabla^2 u(0) = \nabla^3 u(0) =0.
\end{equation}
\eqref{eqn:pde3} implies that $\varphi(x)=x+u(x)$ is a harmonic map, i.e. the coordinate given by $\varphi$ is a Bianchi coordinate and it follows from \eqref{eqn:zerou} that the expansion in this new Bianchi  coordinate remains the same.

To solve \eqref{eqn:pde3}, we use Lemma \ref{lem:pde} with $u_0=0$. Therefore, it remains to check that
\[
	\Gamma^i_{aa}(0)=0; \quad \partial_l \Gamma^i_{aa}(0)=0.
\]
The first equation is trivial because there is no linear term in the expansion. For the second term, we compute
\begin{eqnarray*}
	\Gamma^i_{jk}(x)&=& -\frac{1}{3}\left( R_{jkbi} + R_{jbki} + R_{kjbi}+R_{kbji}-R_{jibk}-R_{jbik} \right) x_b + \cdots 
\end{eqnarray*}
Hence,
\[
	\partial_{l}\Gamma^i_{aa} (0)= \frac{2}{3} (Ric)_{li}(0)=0,
\]
where we have used the assumption $Ric(\bx)=0$.

This finishes the proof of Theorem \ref{thm:known} in the Bianchi coordinate case.  For the harmonic coordinate case, it suffices to replace \eqref{eqn:pde3} with 
\[
	g^{ab} \left(\partial_{ab} u_i - \Gamma^k_{ab}(\delta_{ik} + \partial_k u_i) \right)	=0.
\]
The rest of the proof is the same.

\section{Coordinate at infinity of asymptotically flat Einstein metric}
\label{sec:infinity}

In this section, we prove Theorem \ref{thm:main1} and Theorem \ref{thm:main2}. We start by recalling the famous result of Bando, Kasue and Nakajima \cite{bando1989on}. Let $(M,g)$ be as in Theorem \ref{thm:main1}. It was proved in \cite{bando1989on} that there is a compact set $K$ and a coordinate defined on the universal cover $\Omega\setminus K$ in which the metric satisfies
\[
	\abs{g_{ij}(x)-\delta_{ij}} \leq C\frac{1}{r^{m-1-\varepsilon}} \qquad \text{on} \quad \Real^m\setminus B_R.
\]
\begin{rem}
	Throughout this section, we assume that $\Omega\setminus K$ is simply connected for simplicity. In fact, we work on $\Omega_R$ and in case the fundamental group $\Gamma$ is not trivial, we should insist that every function/tensor be $\Gamma$-invariant/equivariant.
\end{rem}
Moreover, the coordinate function $x_i$ is harmonic, i.e.
\[
	\triangle_g x_i =0.
\]
Therefore, the Ricci-flatness implies
\begin{equation}
	\label{eqn:ricciflat}
	\triangle g =Q(g,\partial g)
\end{equation}
where $Q$ is quadratic in $\partial g$.

Due to \eqref{eqn:ricciflat} and the elliptic estimate, we know for any $k\in \mathbb N$,
\[
	\abs{ \partial^{(k)} g_{ij}(x)} \leq C(k) \frac{1}{\abs{x}^{m+k-1-\varepsilon}}.
\]
By Definition \ref{defn:O}, we may write
\begin{equation}
	\label{eqn:maywrite}
	g_{ij}=\delta_{ij}+ O(\abs{x}^{-(m-1-\varepsilon)}).
\end{equation}
Hence, the right hand side of \eqref{eqn:ricciflat} is in $O(\abs{x}^{-2(m-\varepsilon)})$. By Theorem \ref{thm:linear}, we know the existence of a matrix of harmonic functions $h_{ij}$ such that
\[
	g_{ij}=h_{ij}+ O(\abs{x}^{-2(m-\varepsilon)+2}).
\]
As harmonic functions defined on $\Real^m\setminus B_R$, $h_{ij}$ has a well-known expansion. By comparing with \eqref{eqn:maywrite}, we obtain a set of coefficients $A_{ijk}$ (with $A_{ijk}=A_{jik}$) such that
\begin{equation}
	\label{eqn:gpre1}
	g_{ij}=\delta_{ij}+A_{ijk}\frac{x_k}{\abs{x}^{m}} + O(\abs{x}^{-m}).
\end{equation}

As discussed in Section \ref{sec:intro}, when the metric $g$ admits some nontrivial isometric action of a finite group $\Gamma$, it is known that $A_{ijk}$ must vanish. Here we do not assume this additional symmetry. Instead of proving the vanishing of $A_{ijk}$ directly, we show the existence of {\it another} coordinate in which the metric expansion has vanishing $m-1$ order term.

\subsection{A coordinate with vanishing $m-1$ order term}

Using notations given above, we prove
\begin{prop}
	\label{prop:step1} There is a Bianchi coordinate, denoted by $\tilde{x}$, defined on $\Real^m \setminus B_{R'}$ (for some $R'>R$) such that the metric has an expansion
	\[
		\tilde{g}_{ij}=\delta_{ij} +O (\abs{\tilde{x}}^{-m+\varepsilon})\qquad \text{for small} \quad \varepsilon>0,
	\]
	where $\tilde{g}_{ij}$ is the metric matrix of $g$ in coordinate $\tilde{x}$.
\end{prop}

\begin{proof}
	Since the coordinate functions $x_i$'s that we use in \eqref{eqn:gpre1} are harmonic, the coefficients $A_{ijk}$ satisfy some additional requirements. More precisely, the harmonic function equation $\triangle_g x_i =0$ implies that
	\[
		g^{kl} \Gamma^i_{kl}(x) =0.
	\]
	Plugging \eqref{eqn:gpre1} into the above equation and ignoring all terms that decay faster than $\abs{x}^{-m}$, we obtain
	\[
		2 \partial_{x_a} \left( A_{iab}\frac{x_b}{\abs{x}^m} \right) - \partial_{x_i} \left( A_{aab}\frac{x_b}{\abs{x}^m} \right)=0,
	\]
	which is equivalent to
	\[
		(2A_{iaa}\delta_{bc} - 2mA_{ibc}- A_{aai}\delta_{bc}+m A_{aab}\delta_{ic}) x_bx_c=0.
	\]
	By the definition of ${\mathbf B}_1$ in Section \ref{subsec:maps}, the above equation becomes
	\[
		{\mathbf B}_1(A)=0.
	\]
	Proposition \ref{prop:psi3} implies the existence of $B\in \Real^m$ such that $\Psi_3(B)=-A$, i.e.
	\[
		A_{ijk}=(m-2)(\delta_{ik}B_j + \delta_{jk}B_i).
	\]
	With this $B$, we consider a coordinate change $\varphi$ given by
	\[
		(\varphi(x))_i= x_i+ \frac{B_i}{\abs{x}^{m-2}}.
	\]
	Note that when $\abs{x}$ is large, $\varphi$ is a diffeomorphism. The pullback metric $\varphi^*(g)$ has an expansion
	\begin{equation}
		\label{eqn:trans1}
		\begin{split}
			(\varphi^* g)_{pq} &=g_{ij}(\varphi(x)) \left( \delta_{ip} - (m-2)\frac{B_ix_p}{\abs{x}^{m}} \right) \left( \delta_{jq} - (m-2) \frac{B_j x_q}{\abs{x}^{m}} \right) \\
					   &= \delta_{pq} + \abs{x}^{-m} \left( A_{pqk}x_k - (m-2)B_q x_p -(m-2) B_px_q \right) + O(\abs{x}^{-m}) \\
					   &= \delta_{pq}+ O(\abs{x}^{-m}).
		\end{split}
	\end{equation}

	To finish the proof of Proposition \ref{prop:step1}, we need to apply Theorem \ref{thm:pde}. Using the new coordinate given by $\varphi$ (still denoted by $x$), the metric $g$ satisfies
	\begin{equation}
		\label{eqn:gijgood}
		g_{ij}=\delta_{ij}+O(\abs{x}^{-m}).
	\end{equation}
	Hence, 
	\[
		\Gamma^i_{jk}(x) = O(\abs{x}^{-(m+1)}).
	\]
	It then follows that the identity map $\varphi(x)=x$ satisfies approximately the harmonic map equation in the sense that
	\begin{equation}
		\label{eqn:bian}
		\triangle \varphi_i + \Gamma^i_{jk}(\varphi) \pfrac{\varphi_j}{x_a} \pfrac{\varphi_k}{x_a}  = O(\abs{x}^{-(m+1)}).
	\end{equation}
	Theorem \ref{thm:pde} then implies that there is a solution $\tilde{\varphi}(x)$ to the above equation such that
	\[
		\tilde{\varphi}(x)= x + O(\abs{x}^{-(m-1-\varepsilon)}).
	\]
	Using a computation similar to \eqref{eqn:trans1} again, we obtain the desired expansion in this proposition in the coordinate given by $\tilde{\varphi}$. The new coordinate now is Bianchi. 
\end{proof}

\begin{rem}
	There is a similar version of Proposition \ref{prop:step1} that gives a harmonic coordinate instead of a Bianchi one. Here we give an indication of how to modify the proof. The proof before \eqref{eqn:gijgood} remains intact. We observe that
	\[
		\triangle_g x_a = g^{ij}(\partial^2_{ij}x_a-\Gamma^k_{ij}(x)\partial_k x_a)=- g^{ij}(x)\Gamma_{ij}^a(x)=O(\abs{x}^{-(m+1)}).
	\]
	We need to find $f_a\in O(\abs{x}^{-(m-1-\varepsilon)})$ so that
	$$
	\triangle_g f_a = \triangle_g x_a,
	$$
	that is
	\[
		\triangle f_a = \triangle_g x_a + g^{ij}(x)\Gamma^k_{ij}(x)\partial_k f_a - (g^{ij}-\delta_{ij})\partial^2_{ij}f_a.
	\]
	This could be done by a similar fixed point argument as in the proof of Theorem \ref{thm:pde}. Given $f_a$, it suffices to check that $\tilde{x}_a=x_a-f_a$ is the desired coordinate.
\end{rem}

\subsection{Proof of Theorem \ref{thm:main1}}

Recall that the Ricci curvature formula in a general coordinate system is
\[
	Ric(g)_{ij}= \frac{1}{2}\sum_{s,t}g^{st} \left( \frac{\partial^2 g_{is}}{\partial x_j \partial x_t} + \frac{\partial^2 g_{js}}{\partial x_i \partial x_t} - \frac{\partial^2 g_{ij}}{\partial x_s \partial x_t} - \frac{\partial^2 g_{st}}{\partial x_i \partial x_j}\right) + Q(g,\partial g).
\]
By the definition of Bianchi coordinate, the Ricci-flat equation in Bianchi coordinate becomes
\begin{equation}
	\label{eqn:bianchirf}
	\triangle g_{ij} + Q(g,\partial g) + Q'(g-\delta,\partial^2 g)=0,
\end{equation}
where $Q$ is quadratic in $\partial g$ and $Q'$ is linear in both $g-\delta$ and $\partial^2 g$.

Given Proposition \ref{prop:step1}, by the same argument that we used to prove \eqref{eqn:gpre1}, we obtain 
\begin{equation}
	\label{eqn:step2ass}
	g_{ij}= \delta_{ij}+A_{ijkl}\frac{x_kx_l}{\abs{x}^{m+2}}+ O(\abs{x}^{-(m+1)}).
\end{equation}
Here $A_{ijkl}$ is a tensor satisfying the obvious symmetry
$$
A_{ijkl}=A_{jikl}=A_{ijlk}.
$$
Plugging \eqref{eqn:step2ass} into \eqref{eqn:bianchirf}, focusing on the leading term and taking the limit $\abs{x}\to \infty$, we obtain
\begin{equation}
	\label{eqn:step2hf}
	\triangle \left( A_{ijkl}\frac{x_kx_l}{\abs{x}^{m+2}} \right)=0,
\end{equation}
which implies that $A_{ijkk}=0$, or equivalently,
\[
	A\in S^2(\Real^m)\otimes S^2_0(\Real^m).
\]

In addition to \eqref{eqn:step2hf}, being a Bianchi coordinate puts strong restrictions to the $A$ in \eqref{eqn:step2ass}. The Bianchi condition means
\[
	\partial_j g_{ij} -\frac{1}{2} \partial_i g_{jj}=0 \qquad \text{for all} \quad x.
\]
Plugging \eqref{eqn:step2ass} into the above equation  and focusing only on the $\abs{x}^{-(m+1)}$ order term, we obtain
\[
	\left( 4A_{aial}\delta_{kj} -2(m+2)A_{jikl}-2A_{aail}\delta_{kj}+(m+2)A_{aakl}\delta_{ij} \right)x_kx_lx_j=0.
\]
Note that we can always do the symmetrization over the indices $k,l,j$ for the coefficients before $x_kx_lx_j$. This leads us to the definition of ${\mathbf B}_2$ in Section \ref{sec:alg}
\[
	{\mathbf B}_2: S^2(\Real^m)\otimes S^2(\Real^m)\to \Real^m \otimes S^3(\Real^m)
\]
by
\[
	{\mathbf B}_2(A)= S_{(234)}[4A_{(a1a4)}\delta_{(23)}-2(m+2)A_{(3124)}-2A_{(aa14)}\delta_{(23)}+(m+2)A_{(aa24)}\delta_{(13)}].
\]
In summary, we have shown that the tensor $A$ in \eqref{eqn:step2ass} lies in
\[
	Z:=S^2(\Real^m)\otimes S^2_0(\Real^m) \cap {\rm Ker}({\mathbf B}_2).
\]
To finish the proof of Theorem \ref{thm:main1}, we {\bf claim} the existence of {\it another} Bianchi coordinate (which is a small perturbation of the one given in Proposition \ref{prop:step1}) such that in the new coordinate (denoted by $\tilde{x}$) we have
\begin{equation}
	\label{eqn:1goal}
	\tilde{g}_{ij}=\delta_{ij}+ \tilde{W}_{ijkl}\frac{\tilde{x}_k\tilde{x}_l}{\abs{\tilde{x}}^{m+2}} + O(\abs{\tilde{x}}^{-(m+1)})
\end{equation}
for some $\tilde{W}\in \widetilde{\mathcal W}$. Given $\tilde{W}$, we can find $W\in \mathcal W $ such that
\[
	s(W)=\frac{1}{2}[W_{(1342)}+W_{(1432)}]=\tilde{W}.
\]
Obviously,
\[
	\tilde{W}_{ijkl}\frac{\tilde{x}_k\tilde{x}_l}{\abs{\tilde{x}}^{m+2}} = W_{iklj}\frac{\tilde{x}_k\tilde{x}_l}{\abs{\tilde{x}}^{m+2}}.
\]
Hence the proof of Theorem \ref{thm:main1} is done. 

It remains to prove the claim. The strategy is the same as before. We first find a vector-valued polynomial function
\begin{equation}
	\label{eqn:uiB}
	u_i= B_{ij}\frac{x_j}{\abs{x}^{m}}
\end{equation}
such that $\varphi(x)=x+u(x)$ defines a diffeomorphism near the infinity and the pullback metric $\varphi^*g$ has the required expansion. Second, we solve the harmonic map equation to get a Bianchi coordinate with the same expansion.

For the first step, if $g$ is given in \eqref{eqn:step2ass} and the coordinate transformation $\varphi$ is given by $B$ as above, we derive a formula of $\varphi^*(g)$
\begin{eqnarray*}
	(\varphi^*g)_{pq} &=& \left( \delta_{ij} +A_{ijkl}\frac{\varphi_k \varphi_l}{\abs{\varphi}^{m+2}} + O(\abs{x}^{-(m+1)}) \right) \\
			  && \cdot \left( \delta_{ip} + B_{ia}(\frac{\delta_{ap}}{\abs{x}^m} - m \frac{x_ax_p}{\abs{x}^{m+2}}) \right) \\
			  && \cdot \left( \delta_{jq} + B_{ja}(\frac{\delta_{aq}}{\abs{x}^m} - m \frac{x_ax_q}{\abs{x}^{m+2}}) \right).
\end{eqnarray*}
The $\abs{x}^{-m}$ order term in the right hand side of the above is 
\[
	\left(A_{pqbc}+B_{qp}\delta_{bc}+B_{pq}\delta_{bc}-\frac{m}{2} (B_{qb}\delta_{pc}+B_{qc}\delta_{pb}+B_{pb}\delta_{qc}+B_{pc}\delta_{qb})\right) x_bx_c.
\]
This explains the definition of $\Psi_4$ in Section \ref{sec:alg}
\[
	\Psi_4: (\Real^m)^{\otimes 2} \to S^2(\Real^m)\otimes S^2(\Real^m)
\]
defined by
\[
	\Psi_4(B)= [B_{(12)}\delta_{(34)}+B_{(21)}\delta_{(34)}- \frac{m}{2}(B_{(13)}\delta_{(24)}+B_{(14)}\delta_{(23)}+B_{(23)}\delta_{(14)}+B_{(24)}\delta_{(13)})].
\]
\begin{lem}
	\label{lem:findB}
	For any $A\in Z$, there exists $B\in (\Real^m)^{\otimes 2}$ satisfying
	\[
		A+ \Psi_4(B)\in \widetilde{\mathcal W}.
	\]
	Moreover, this $B$ is unique.
\end{lem}
\begin{proof}
	It is exactly Proposition \ref{prop:decomp2}.
\end{proof}

With the new coordinate defined by $\varphi(x)=x+u(x)$ where $u$ is defined in \eqref{eqn:uiB} and $B$ is given in Lemma \ref{lem:findB}, we have
\begin{equation}
	\label{eqn:step2almost}
	g_{ij}=\delta_{ij}+ W_{ijkl}\frac{x_kx_l}{\abs{x}^{m+2}} + O(\abs{x}^{-(m+1)})
\end{equation}
where $W\in \widetilde{\mathcal W}$. This is the first step.\\

Our next step is to turn this coordinate into a Bianchi coordinate with a small perturbation that keeps the expansion in \eqref{eqn:step2almost}. For this purpose,  we check that the identity map satisfies the harmonic map equation up to high order.

\begin{lem}
	\label{lem:idapprox}
	The identity map $\varphi(x)=x$ is approximately harmonic map from $g_e$ to $g$ in the sense that
	\begin{equation}
		\label{eqn:almosthm}
		\triangle \varphi_i + \Gamma^i_{jk}(\varphi) \pfrac{\varphi_j}{x_a} \pfrac{\varphi_k}{x_a}  = O(\abs{x}^{-(m+2)}).
	\end{equation}
	Here $\Gamma$ is the Christoffel symbol of $g$.
\end{lem}
\begin{proof}
	Since $\varphi(x)=x$, \eqref{eqn:almosthm} is the same as
	\[
		\abs{\Gamma^i_{aa}}(x) = O(\abs{x}^{-(m+2)}).
	\]
	With \eqref{eqn:step2almost}, we can compute
	\[
		\pfrac{g_{ij}}{x_k}= W_{ijbc} \left( 2 \frac{\delta_{kb}x_c}{\abs{x}^{m+2}} - (m+2) \frac{x_bx_cx_k}{\abs{x}^{m+4}} \right) + O(\abs{x}^{-(m+2)}).
	\]
	Hence
	\begin{eqnarray*}
		\Gamma^i_{jk} &=& \frac{1}{2} \left( \pfrac{g_{ij}}{x_k} + \pfrac{g_{ik}}{x_j}- \pfrac{g_{jk}}{x_i} \right) + O(\abs{x}^{-(m+2)})  \\
			      &=& \frac{1}{2\abs{x}^{(m+4)}} \left( 2W_{ijbc}  \delta_{kb} x_c \abs{x}^2 - (m+2)W_{ijbc}x_bx_cx_k \right)  \\
			      && + \frac{1}{2\abs{x}^{(m+4)}} \left( 2W_{ikbc}  \delta_{jb} x_c \abs{x}^2 - (m+2)W_{ikbc}x_bx_cx_j \right)  \\
			      && - \frac{1}{2\abs{x}^{(m+4)}} \left( 2W_{jkbc}  \delta_{ib} x_c \abs{x}^2 - (m+2)W_{jkbc}x_bx_cx_i \right)  + O(\abs{x}^{-(m+2)})\\
			      &=& \frac{1}{\abs{x}^{m+4}} \left[ (W_{ijkc}+W_{ikjc}-W_{jkic})\delta_{ab}\right] x_ax_bx_c \\
			      &&- \frac{m+2}{2\abs{x}^{m+4}}(W_{ijbc}\delta_{ka} +W_{ikbc}\delta_{ja} - W_{jkbc}\delta_{ia})  x_ax_bx_c + O(\abs{x}^{-(m+2)}).
	\end{eqnarray*}
	Since all traces of $W$ vanish, when we take the $(j,k)$ trace of the above equation for $\Gamma^i_{jk}$, we obtain
	\[
		\Gamma^i_{aa}=- \frac{m+2}{2\abs{x}^{m+4}}(W_{iabc} +W_{iabc})  x_ax_bx_c + O(\abs{x}^{-(m+2)}).
	\]
	The proof of the lemma is done by noticing that 
	$$
	W_{iabc}x_ax_bx_c=0.
	$$
	In fact, since $W\in \widetilde{\mathcal W}$, by its definition in Section \ref{sec:alg}, we have the first Bianchi identity:
	$$
		W_{iabc}+W_{ibca}+W_{icab}=0.
	$$
\end{proof}
Given this lemma, we obtain from Theorem \ref{thm:pde} that there exists a Bianchi coordinate in which 
$$
	g_{ij}=\delta_{ij}+ W_{ijkl}\frac{x_kx_l}{\abs{x}^{m+2}} + O(\abs{x}^{-(m+1-\varepsilon)})
$$
for any $\varepsilon>0$. We then use the Einstein equation as in the beginning of this section to show \eqref{eqn:step2almost}. This completes the proof of Theorem \ref{thm:main1}.

\subsection{Proof of Theorem \ref{thm:main2}}

Let $(M,g)$ be as assumed in Theorem \ref{thm:main1}. It follows from the theorem that we have a coordinate defined on $\Omega_R$ and a Weyl-type tensor $W$ such that
\[
	g_{ij}=\delta_{ij}+ W_{iabj} \frac{x_ax_b}{\abs{x}^{m+2}} + O(\abs{x}^{-(m+1)}).
\]
The aim of this section is to show that $W$ (as a tensor) is an intrinsic geometric property of $(M,g)$. This follows from Theorem \ref{thm:main2}, because if we have some other coordinate in which we have a similar expansion, the two Weyl-type tensors are the same (up to a rotation in $\Real^m$).

\begin{rem}
	We do not claim that the coordinate in Theorem \ref{thm:main1} is unique. In fact, it is not.
\end{rem}

\begin{proof}[Proof of Theorem \ref{thm:main2}]
	Instead of one metric and two coordinates, we regard	
	\[
		g= (\delta_{ij}+W_{iabj}\frac{x_ax_b}{\abs{x}^{m+2}}+O(\abs{x}^{-(m+1)})) dx_i\otimes dx_j	
	\]
	as a metric defined on $\Omega_R$ and
	\[
		\tilde{g}= (\delta_{ij}+\tilde{W}_{iabj}\frac{y_ay_b}{\abs{y}^{m+2}}+O(\abs{y}^{-(m+1)})) dy_i\otimes dy_j.
	\]
	as another metric defined on $\Omega_{R'}$ in another copy of $\Real^m$. Then we have a map $\varphi:\Omega_R\to\Omega_{R'}\subset \Real^m$ such that $\varphi^*(\tilde{g})=g$. Indeed, this is just the identity map with its domain and its range represented by different coordinates.

	In particular, $\varphi$ is a harmonic map. Moreover, it follows from $\varphi^*(\tilde{g})=g$ that
	\begin{equation}
		\label{eqn:boundgradient}
		\abs{\pfrac{\varphi_i}{x_j}}\leq C \qquad \text{on} \quad \Omega_R.
	\end{equation}
	Since both $x$ and $y$ are coordinates of the same end, we have
	\[
		\lim_{\abs{x}\to \infty } \abs{\varphi(x)} =\infty.
	\]
	By \eqref{eqn:boundgradient}, for any $x,x'$ in $\Omega_R$,
	\[
		\abs{\varphi(x)-\varphi(x')}\leq C \abs{x-x'}.
	\]
	Hence, by the symmetry of $x$ and $y$ in the statement of the theorem, we have
	\begin{equation}
		\label{eqn:compare}
		C<\frac{\abs{\varphi(x)}}{\abs{x}}\leq C'\qquad \text{on} \quad \Omega_R.
	\end{equation}

	Let $\Gamma$ and $\tilde{\Gamma}$ be the Christoffel symbols of $g$ and $\tilde{g}$ respectively. The harmonic map equation of $\varphi$ is
	\begin{equation}
		\label{eqn:hme}
		g^{ij}\left( \partial_i \partial_j\varphi_a  - \Gamma^k_{ij} \partial_k{\varphi_a} + \tilde{\Gamma}^a_{bc}(\varphi)\partial_i \varphi_b \partial_j \varphi_c \right) =0.
	\end{equation}

	\begin{lem}
		\label{lem:local} The second and the third order derivatives of $\varphi$ are bounded on $\Omega_R$, i.e.
		\[
			\sup_{\Omega_{R+1}} \abs{\partial^2 \varphi}+ \abs{\partial^3 \varphi}\leq C.
		\]
	\end{lem}
	\begin{proof}
		Fix $\alpha\in (0,1)$. For any $x\in \Omega_{R+1}$ and $y\in \Omega_{R'+1}$, we have 
		\[
			\norm{g_{ij}}_{C^\alpha(B(x,1))}, \norm{\Gamma^k_{ij}}_{C^\alpha(B(x,1))}, \norm{\tilde \Gamma^a_{bc}}_{C^\alpha(B(y,1))}\leq C
		\]
		uniformly.
		By \eqref{eqn:boundgradient}, $\tilde{\Gamma}^a_{bc}(\varphi)$ is also bounded uniformly in $C^\alpha(B(x,1))$, i.e.
		\[
			\norm{\tilde \Gamma^a_{bc}\circ \varphi}_{C^\alpha(B(x,1))}\leq C.
		\]
		We apply the $L^p$ estimate to \eqref{eqn:hme} to get $W^{1,p}$ bound for $\partial \varphi$ on $B(x,3/4)$ for any $p$. This gives $C^\beta$($\beta\in (0,1)$) bound for $\partial \varphi$ on $B(x,3/4)$. The lemma then follows from the usual bootstrapping technique.
	\end{proof}
	Here is how we will use this lemma.
	\begin{cor}
		\label{cor:use} For any $\alpha\in (0,1)$, if $r=\abs{x}$, then
		\[
			r^{-2} \partial \varphi, r^{-2} \partial^2 \varphi \in \mathcal X_{\alpha,1}(\Omega_R).
		\]
	\end{cor}
	\begin{proof}
		We prove $r^{-2} \partial^2 \varphi$ only. The other case is the same. By definition, it suffices to show
		\[
			\sup_{\rho\geq R} \rho \norm{ (r^{-2} \partial^2 \varphi)\circ S_\rho}_{C^\alpha(B_2\setminus B_1)}\leq C.
		\]
		To see this,
		\begin{eqnarray*}
			\rho \norm{ (r^{-2} \partial^2 \varphi)\circ S_\rho}_{C^\alpha(B_2\setminus B_1)} &\leq& C \rho^{-1} \norm{(\partial^2 \varphi)\circ S_\rho}_{C^1(B_2\setminus B_1)} \\
													  &\leq& C \sup_{\Omega_{R}} (\abs{\partial^2 \varphi}+ \abs{\partial^3 \varphi}) \leq C.
		\end{eqnarray*}
	\end{proof}
	We rewrite \eqref{eqn:hme} as
	\begin{equation}
		\label{eqn:newhme}
		\triangle \varphi_a = (\delta_{ij}-g^{ij}) \partial^2_{ij} \varphi_a + g^{ij}\Gamma^k_{ij} \partial_k \varphi_a - g^{ij} \tilde{\Gamma}^a_{bc}(\varphi) \partial_i \varphi_b \partial_j \varphi_c.
	\end{equation}
	We denote the three terms in the right hand side by $I$, $II$ and $III$ respectively and find that
	\begin{eqnarray*}
		I&=&  (r^{2}(\delta_{ij}-g^{ij})) \cdot (r^{-2} \partial^2_{ij} \varphi_a) \\
		II&=&  (r^{2}g^{ij}\Gamma^k_{ij}) \cdot (r^{-2} \partial_k \varphi_a) \\
		III &=& (r^{4 }g^{ij} \tilde{\Gamma}^a_{bc}(\varphi))\cdot (r^{-2} \partial_i \varphi_b)\cdot (r^{-2} \partial_j \varphi_c).
	\end{eqnarray*}
	On one hand, Corollary \ref{cor:use} implies that $r^{-2}\partial \varphi$ and $r^{-2} \partial^2 \varphi$ are in $\mathcal X_{\alpha,1}$. On the other hand, our assumptions on $g$ together with the fact that $m+1\geq 5$, imply that 
	\begin{equation}
		\label{eqn:easier}
		(r^{2}(\delta_{ij}-g^{ij})), (r^{2}g^{ij}\Gamma^k_{ij})\in \mathcal X_{\alpha,2}.
	\end{equation}
	We also claim that
	\begin{equation}
		\label{eqn:harder}
		(r^{4}g^{ij} \tilde{\Gamma}^a_{bc}(\varphi)) \in \mathcal X_{\alpha,1}.
	\end{equation}
	To see this, it suffices to show that
	\[
		\tilde{\Gamma}^a_{bc}(\varphi) \in \mathcal X_{\alpha,5}.
	\]
	For any $\rho\geq R$, we estimate $\tilde{\Gamma}^a_{bc}\circ \varphi \circ S_\rho$. For $x\in B_2\setminus B_1$,
	\[
		\abs{\tilde{\Gamma}^a_{bc}\circ \varphi (\rho x)}\leq C \rho^{-(m+1)}
	\]
	by the expansion of $\tilde{g}$ and \eqref{eqn:compare}. For $x,x'\in B_2\setminus B_1$, assume the segment between $x$ and $x'$ lies in $B_2\setminus B_1$ and $\xi$ is the point on the segment for which the mean value inequality holds, we have
	\begin{eqnarray*}
	&&\abs{\tilde\Gamma^a_{bc}\circ \varphi(\rho x) - \tilde \Gamma^a_{bc}\circ \varphi(\rho x')} \\
	&\leq& \abs{\partial_y \tilde \Gamma^a_{bc}(\varphi (\rho\xi)) }\abs{\partial \varphi} \rho \abs{x-x'} \\
	&\leq& C \rho^{-(m+1)} \abs{x-x'}.
	\end{eqnarray*}
	This finishes the proof of the claim.

	Using \eqref{eqn:easier} and \eqref{eqn:harder}, we find that for some $\alpha_1\in (0,1)$ and some $\beta_1\in (2,3)$,  the right hand side of \eqref{eqn:newhme} is in $\mathcal X_{\alpha_1,\beta_1}$.  

	Theorem \ref{thm:linear} implies the existence of some harmonic function $h_a$ such that
	\[
		\varphi_a - h_a \in \mathcal X_{\alpha_1+2,\beta_1-2}.
	\]
	By \eqref{eqn:boundgradient}, $h_a$ grows at most linearly at the infinity and hence we may assume $h_a=A_{ai}x_i+ c_a$ for some constant matrix $(A_{ai})$ and constant vector $c=(c_a)$ (by neglecting terms decaying faster than $\beta_1-2$). In terms of $\varphi$, we have proved
	\begin{equation}
		\label{eqn:preexp}
		\varphi_a(x)= A_{ai}x_i+ c_a + e_{\alpha_1+2,\beta_1-2} \qquad \text{for some} \quad e_{\alpha_1+2,\beta_1-2}\in \mathcal X_{\alpha_1+2,\beta_1-2}.
	\end{equation}

	Instead of $x_k$, we consider a new coordinate given by $\tilde{x}_k=x_k + (A^{-1}c)_k$, the metric tensor has the same expansion with the same coefficients
	\[
		g= (\delta_{ij}+W_{iabj}\frac{\tilde{x}_a \tilde{x}_b}{\abs{\tilde{x}}^{m+2}} + O(\abs{\tilde{x}}^{-(m+1)})) d\tilde{x}_i \otimes d\tilde{x}_j.
	\]
	Since the conclusion of Theorem \ref{thm:main2} cares only about $W$ and $\tilde{W}$. Hence, we may assume without loss of generality that $c_a=0$ in \eqref{eqn:preexp}.

	Moreover, by taking the limit $x\to \infty $ in $g_{ij}(x)=\tilde{g}_{ab}(\varphi(x))\pfrac{\varphi_a}{x_i}\pfrac{\varphi_b}{x_j}$, we know that $A_{ai}$ is an orthogonal matrix, as claimed by the theorem. By rotating $x$, or $y$, we may assume it is the identity matrix. The proof of the theorem is then reduced to proving
	\[
		W=\tilde{W}.
	\]
	To see this, we need to study higher regularity of $\varphi$ at the infinity. We set (by \eqref{eqn:preexp})
	\[
		u(x)=\varphi(x)-x \in \mathcal X_{\alpha_1+2,\beta_1-2}(\Omega_R).
	\]
	The \eqref{eqn:hme} becomes
	\begin{equation}
		\label{eqn:u}
		\triangle u_a = g^{ij}\Gamma^k_{ij} (\delta_{ak}+\partial_k u_a)- g^{ij}\tilde{\Gamma}^a_{bc}(x+u) (\delta_{bi}+\partial_i u_b)(\delta_{cj}+\partial_j u_c) +(\delta^{ij}-g^{ij}) \partial_i\partial_j u_a.
	\end{equation}

	We claim:
	\begin{equation}
		\label{eqn:equ}
		\triangle u_a \in \mathcal X_{\alpha_1,m+2}.
	\end{equation}
	This is true because we have
	\begin{eqnarray*}
		\partial u &\in& \mathcal X_{\alpha_1+1,\beta_1-1} \\
		\partial^2 u &\in& \mathcal X_{\alpha_1,\beta_1} \\
		g^{ij}-\delta_{ij} &\in& \mathcal X_{*,m}\\
		\Gamma^k_{ii} &\in& \mathcal X_{*,m+2}\\
		\tilde{\Gamma}^a_{bb}(x+u(x)) &\in& \mathcal X_{\alpha_1+2,m+2}.
	\end{eqnarray*}
	Here $*$ can be any positive number.
	The first three lines are obvious. The decay rate of the fourth and the fifth line follows from the fact that the Weyl tensor is traceless and this is exactly what is proved in Lemma \ref{lem:idapprox}. For the fifth line, we also need Lemma \ref{lem:composit}. 

	By \eqref{eqn:equ}, Theorem \ref{thm:linear} implies the existence of $A\in \Real^m$ and $B\in (\Real^m)^{\otimes 2}$ such that
	\[
		u_a = A_a \frac{1}{r^{m-2}}+ B_{ai} \frac{x_i}{\abs{x}^m}+ e_{\alpha_1+2,m-\varepsilon}.
	\]
	by an argument used in the proof of \eqref{eqn:preexp}. Here $e_{\alpha_1+2,m-\varepsilon}$ is a function in $\mathcal X_{\alpha_1+2,m-\varepsilon}(\Omega_R)$ for some small $\varepsilon>0$.

	Since $\varphi=x+u(x)$ is the coordinate change, we have 
	\[
		g_{ij} = (\delta_{ai}+ \partial_i u_a) \tilde{g}_{ab}(x+u(x)) (\delta_{bj}+ \partial_j u_b).
	\]
	By comparing the $\abs{x}^{1-m}$ order term of both sides, we find that
	\[
		A=0.
	\]
	By comparing the $\abs{x}^{-m}$ order term, we get
	\[
		s(W)=s(\tilde{W})+ \Psi_4(B).
	\]
	For the definition of $\Psi_4$ and this computation, we refer to the computation before Lemma \ref{lem:findB}. It follows from Lemma \ref{lem:findB} that the intersection of the image of $\Psi_4$ and $\widetilde{{\mathcal W}}$ is trivial. Hence, we must have $W=\tilde{W}$ and $B=0$.
\end{proof}
In the above proof, we have obtained the following {\it almost uniqueness} result.
\begin{cor}
	\label{cor:apply}
	Let $x,y$ be the coordinates in Theorem \ref{thm:main2}. There exist an orthogonal matrix $A$ and vector $c\in \Real^m$ such that
	\[
		y_i= A_{ij}x_j + c_i + e_{\alpha_1+2,m-\varepsilon} \qquad \text{for} \quad i=1,\cdots ,m
	\]
	for some $\alpha_1\in (0,1)$, small $\varepsilon>0$ and
	\[
		e_{\alpha_1+2,m-\varepsilon}\in \mathcal X_{\alpha_1+2,m-\varepsilon}.
	\]
\end{cor}

\section{Renormalized volume}
\label{sec:application}

In this section, we generalize the renormalized volume defined by \cite{biquard2023renormalized} for four dimensional ALE to general dimensions. The definition has been given in Section \ref{sec:intro}. The goal of this section is to prove Theorem \ref{thm:volume}. We also show that in dimension four, this definition coincides with the one defined in \cite{biquard2023renormalized}.

\subsection{The definition}
\label{sub:welldefine}
Assume that $(M,g)$ satisfies the assumptions in Theorem \ref{thm:main1} and for simplicity assume that $\Omega$ is the only end. Notice that the renormalized volume depends on the finite group $\Gamma$ and hence in this section, we do not assume that $\Gamma$ is trivial. The definition itself depends only on the geometry of the end.

Using the Bianchi coordinate $x$ given by Theorem \ref{thm:main1}, we have the expansion of the metric $g$ (lifted to the universal cover of $\Omega\setminus K$, still denoted by $g$)
\begin{equation}
	\label{eqn:expx}
	g_{il}=\delta_{il}+W_{ijkl}\frac{x_jx_k}{\abs{x}^{m+2}} +O(\abs{x}^{-(m+1)}).
\end{equation}
Set 
\[
	B_r=M \setminus \set{\abs{x}>r}	 \qquad \text{and} \qquad \tilde{B}_r= \set{x\in \Real^m|\, \abs{x}<r}/\Gamma.
\]
The renormalized volume $\mathcal V$ is defined to be the limit
\[
	\lim_{r\to \infty } V_g(B_r)-V_{g_e}(\tilde{B}_r),
\]
which we claim to exist.
To see this, we derive from \eqref{eqn:expx} that
\begin{equation}
	\label{eqn:final2}
	\begin{split}
		dV_g =& \sqrt{{\rm det} g} dx \\
		=& \left(1 + \frac{1}{2}{\rm Tr}(W_{ijkl}\frac{x_jx_k}{\abs{x}^{m+2}})+ O(\abs{x}^{-(m+1)})\right) dV_{g_e} \\
		=& (1+ O(\abs{x}^{-(m+1)}))dV_{g_e}.
	\end{split}
\end{equation}
Hence, for any $s>r>0$,
\begin{eqnarray*}
	\abs{\int_{B_s\setminus B_r} dV_g - \int_{\tilde{B}_s\setminus \tilde{B}_r} dV_{g_e}} &=& {\int_{\tilde{B}_s\setminus \tilde{B}_r} O(\abs{x}^{-(m+1)}) dx} \\
											      &\leq& C \abs{\int_r^s \frac{1}{t^2} dt}.
\end{eqnarray*}
The proof of the claim is done.

To justify the definition, we need to show that if $y$ is another coordinate at the infinity satisfying the assumptions of Theorem \ref{thm:main2}, then
\begin{equation}
	\label{eqn:volume}
	\lim_{r\to \infty } V_g(B_r^x)-V_{g_e}(\tilde{B}_r^x) = \lim_{r \to \infty } V_g(B_r^y)-V_{g_e}(\tilde{B}_r^y).
\end{equation}
By Corollary \ref{cor:apply}, we have an orthogonal matrix $A$ and a translation vector $c$ such that
\[
	y=A x + c + e_{\alpha_1+2,m-\varepsilon}
\]
for some $e_{\alpha_1+2,m-\varepsilon}\in \mathcal X_{\alpha_1+2,m-\varepsilon}$. A rotation does not change either $B_r$ or $\tilde{B}_r$, hence we may assume that $A$ is the identity matrix. In order to show that the definition of the renormalized volume is the same for coordinate $x$ and coordinate $y$, it suffices to show
\begin{equation}
	\label{eqn:remains}
	\lim_{r\to \infty } V_g(B_r^x) - V_g(B_r^y)=0.
\end{equation}
Define $\tilde{y}=y-c$. We first claim that
\[
	\lim_{r\to \infty } V_g(B_r^x) - V_g(B_r^{\tilde{y}})=0.
\]
In fact, by definition of $\tilde{y}$, we have
\[
	\abs{x}-C \abs{x}^{-(m-\varepsilon)} \leq \abs{\tilde{y}}\leq \abs{x}+C \abs{x}^{-(m-\varepsilon)}
\]
when $\abs{x}$ is sufficiently large. It follows that
\[
	B^x_{r-Cr^{-(m-\varepsilon)}}\subset B^{\tilde{y}}_r \subset B^x_{r+C r^{-(m-\varepsilon)}},
\]
which implies that 
\[
	\lim_{r\to \infty } V_g(B_r^x)-V_g(B_r^{\tilde y})=0.
\]
Hence, the claim is proved. The proof of \eqref{eqn:remains} is reduced to
\[
	\lim_{r\to \infty } V_g(B_r^{\tilde{y}})-V_g(B_r^y)=0,
\]
or equivalently,
\begin{equation}
	\label{eqn:final1}
	\lim_{r\to \infty } V_g(B_r^{\tilde{y}}\setminus B_r^y)-V_g(B_r^y\setminus B_r^{\tilde{y}})=0.
\end{equation}
Given that $\tilde{y}=y-c$, for sufficiently large $r$, if $g_e$ is the Euclidean metric determined by $y$ (or $\tilde{y}$) outside a compact set, then
\[
	V_{g_e}(B_r^{\tilde{y}} \setminus  B_r^y) + V_{g_e}(B_r^y \setminus B_r^{\tilde{y}}) \leq C r^{m-1}
\]
and
\[
	V_{g_e}(B_r^{\tilde{y}}\setminus B_r^y)-V_{g_e}(B_r^y\setminus B_r^{\tilde{y}})=0.
\]
Given these elementary facts, \eqref{eqn:final1} follows from \eqref{eqn:final2}. This concludes the proof of the first claim in Theorem \ref{thm:volume}.

\subsection{Non-positivity of the renormalized volume}
\label{sub:ricci}
The aim of this subsection is to prove the remaining part of Theorem \ref{thm:volume}. The proof uses a result of \cite{ros1987} (see Remark 3.1 of \cite{biquard2023renormalized}). For $B_r$ defined above, let $H_g$ be the mean curvature of $\partial B_r$ and $dA_g$ be the volume form of $\partial B_r$ with respect to the metric $g$. Since the Ricci curvature of $g$ is non-negative, Theorem 1 in \cite{ros1987} implies
\[
	\int_{\partial B_r}\frac{1}{H_g} dA_g \geq m V_g(B_r).
\]
\begin{rem}
	Theorem 1 in \cite{ros1987} was stated for manifolds instead of orbifolds. However, the proof involves only pointwise computations and integration by parts, which extended to the orbifold setting directly. Computations in the following proof of the rigidity part of Theorem \ref{thm:volume} should be understood similarly.
\end{rem}
Let $H_e$ be the mean curvature of $\partial \tilde{B}_r$ with respect to the flat metric $g_e$, which is just $1/r$. Let $dA_e$ be the standard volume form on $\partial \tilde{B}_r$. Then we have trivially
\[
	\int_{\partial \tilde{B}_r} \frac{1}{H_e} dA_e = m V_e(\tilde{B}_r)
\]
To show that $\mathcal V\leq 0$, it suffices to show
\begin{equation}
	\label{eqn:goal}
	\lim_{r\to \infty } \int_{\partial \tilde{B}_r}\frac{1}{H_e}dA_e - \int_{\partial B_r}\frac{1}{H_g}dA_g=0.
\end{equation}
To see this, we need to study the asymptotics of $\int_{\partial B_r}\frac{1}{H_g}dA_g$ up to high order. Via the coordinate map, we identify $\partial B_r$ with $\partial \tilde{B}_r$. They are the same hypersurface. However, the unit normal, the volume form and the mean curvature depend on the ambient metric. In order to compare these quantities, we start with $\nabla_g r$,
\begin{eqnarray*}
	\nabla_g r &=& (g^{ij} \partial_i r) \partial_j\\
		   &=& (\delta_{ij}- W_{iabj}\frac{x_ax_b}{r^{m+2}} + O(r^{-(m+1)})) \frac{x_i}{r} \partial_j \\
		   &=& \nabla_{g_e} r + O(r^{-(m+1)}).
\end{eqnarray*}
Here in the last line above, we have used the first Bianchi identity of $W$. 
Taking the squared norm, we have
\begin{eqnarray*}
	\abs{\nabla_g r}_g^2 &=& g_{ij} (\frac{x_i}{r})(\frac{x_j}{r}) +O(r^{-(m+1)})\\
			     &=& 1 + O(r^{-(m+1)}).
\end{eqnarray*}
Denote the unit normal vector of $\partial B_r$ (with respect to $g$) by $n_g$ and that of $\partial \tilde{B}_r$ (with respect to $g_e$) by $n_{g_e}$. Since $\nabla_g r$ is perpendicular to $\partial B_r$, we have
\begin{equation}
	\label{eqn:goodng}
	n_g = \frac{\nabla_g r}{ \abs{\nabla_g r}} = n_{g_e} + O(r^{-(m+1)}).
\end{equation}
Recall that we have proved $dV_g=(1+O(r^{-(m+1)}))dV_{g_e}$.
Using the formula
\[
	dA_g = (\iota_{n_g} dV_g)|_{\partial B_r}; \quad dA_{g_e}= (\iota_{n_{g_e}} dV_{g_e})|_{\partial \tilde B_r},
\]
we obtain
\begin{equation}
	\label{eqn:goodAg}
	dA_g= (1+ O(r^{-(m+1)}))dA_{g_e}.
\end{equation}

Recall the formula for the mean curvature $H_g= \frac{1}{m-1}\operatorname{div}_g n_g$, where
\begin{eqnarray*}
\operatorname{div}_g n_g&=&\frac{1}{\sqrt{\operatorname{det} g}} \partial_i\left(\sqrt{\operatorname{det} g}\left(n_g\right)^i\right)  \\
&=&\partial_i (n_g)^i + (n_g)^i \partial_i \log \sqrt{\det g}.
\end{eqnarray*}
On the one hand, the expansion of metric gives
$$
\sqrt{\det g}= 1+ O(r^{-m-1}) \quad \text{and} \quad \partial_i \log \sqrt{\det g} = O(r^{-m-2}).
$$
On the other hand,
\begin{eqnarray*}
	\partial_i(n_g)^i &=& \partial_i(n_e)^i + O(r^{-m-2}) \\
	&=& \frac{m-1}{r} + O(r^{-m-2}) .
\end{eqnarray*}
In summary, we get
\[
	H_g=\frac{1}{r}+ O(r^{-(m+2)})= H_{g_e}+ O(r^{-(m+2)})
\]
and therefore
\begin{equation}
	\label{eqn:goodH}
	\frac{1}{H_g}= \frac{1}{H_{g_e}}+ O(r^{-m})=\frac{1}{H_{g_e}}(1+O(r^{-(m+1)})).
\end{equation}
It is obvious that \eqref{eqn:goal} follows from \eqref{eqn:goodAg} and \eqref{eqn:goodH}. 

In case the renormalized volume $\mathcal V$ is zero, the proof above implies that
\begin{equation}
	\label{eqn:rigidity}
	\lim_{r\to \infty } \left( \int_{\partial B_r} \frac{1}{H_g} dA_g - m V_g(B_r) \right)=0.
\end{equation}
We need to prove that $M$ is isometric to $\Real^m/\Gamma$. The following proof is adapted from the proof of Theorem 1 in \cite{ros1987}. There was a solution $u_r$ satisfying
\[
	\begin{cases}
		\triangle_g u_r =1 & \text{in} \quad B_r \\
		u_r|_{\partial B_r} = 0.&
	\end{cases}
\]

{\bf Claim 1.} Setting
\[
	E_r= \frac{1}{m}(\triangle_g u_r)^2 - \abs{\nabla_g^2 u_r}^2 \leq 0, 
\]
we claim that
\[
	\lim_{r\to \infty } \int_{B_r} E_r dV_g = 0.
\]

{\bf Claim 2.} Setting
\[
	\tilde{u}_r= u_r + \frac{1}{2m}r^2,
\]
we claim that $\tilde{u}_r$ sub-converges locally smoothly to a limit $\tilde{u}_\infty $ that is defined on $M$ and that the traceless Hessian of $\tilde{u}_\infty $ vanishes everywhere.

Assume the claims for now. By definition, $E_r$ remains the same with $u_r$ replaced by $\tilde{u}_r$. Taking the limit, we obtain
\[
	\frac{1}{m}(\triangle_g \tilde{u}_\infty )^2 - \abs{\nabla_g^2 \tilde{u}_\infty }^2 =0 \qquad \text{on} \quad M.
\]
Moreover, by the equation of $u_r$, we have
\[
	\triangle_g \tilde{u}_\infty =1.
\]
In summary, if $X=\nabla_g \tilde{u}_\infty $, we have just proved that
\[
	\mathcal L_X g = \frac{2}{m}g.
	\]
	Namely, $X$ is a conformal Killing vector field. Then $(M,g)$ is isometric to $\Real^m/\Gamma$.

\begin{proof}
	[Proof of Claim 1.] 
	Using the Reilly's formula
	\[
		\int_{B_r} (\triangle_g u_r)^2 - \abs{\nabla_g^2 u_r}^2 dV_g = \int_{\partial B_r} (m-1) H_g \abs{\partial_{n_g} u_r}^2 dA_g,
	\]
	we have
	\[
		\frac{m-1}{m} V_g(B_r)	+ \int_{B_r} E_r dV_g = (m-1) \int_{\partial B_r} H_g \abs{\partial_{n_g} u_r}^2 dA_g.
	\]
	Following \cite{ros1987},
	\begin{eqnarray*}
		V_g(B_r)^2 &=& \left( \int_{\partial B_r} \partial_{n_g} u_r dA_g \right)^2 \\
			   &\leq& \int_{\partial B_r} H_g \abs{\partial_{n_g} u_r}^2 dA_g \cdot \int_{\partial B_r} \frac{1}{H_g} dA_g \\
			   &=& \left( \frac{1}{m} V_g(B_r) + \frac{1}{m-1}\int_{B_r} E_r dV_g \right) \left( m V_g(B_r) + o(1) \right).
	\end{eqnarray*}
	Here in the last line above, we used \eqref{eqn:rigidity}.
	Claim 1 follows from the above inequality and the fact that $E_r\leq 0$.
\end{proof}

\begin{proof}
	[Proof of Claim 2.]
	While the coordinate is only defined at the infinity, we can extend the domain of the function $\frac{1}{2m}\abs{x}^2$ smoothly inside $M$. Denote the new function by $u_a$. Given the expansion of the metric at the infinity, we may compute
	\begin{equation}
		\label{eqn:goodua}
		\triangle_g u_a = 1 + O(\abs{x}^{-(m+\varepsilon)}).
	\end{equation}
	Hence, for any $r$ sufficiently large, we have
	\begin{equation}
		\label{eqn:gooddiff}
		\begin{cases}
			\triangle_g (\tilde u_r- u_a) = 1- \triangle_g u_a = O(\abs{x}^{-(m+\varepsilon)}) & \text{in} \quad B_r \\
			(\tilde u_r-u_a) = 0 & \text{on} \quad \partial B_r.
		\end{cases}	
	\end{equation}
	Multiplying both sides by $\tilde{u}_r - u_a$ and integrating by parts, we get
	\[
		\int_{B_r} \abs{\nabla_g (\tilde{u}_r -u_a)}^2 dV_g = \int_{B_r} (\triangle_g u_a-1) (\tilde{u}_r-u_a)dV_g.
	\]
	Since we have uniform Sobolev inequality for $B_r$ (due to the Ricci lower bound, volume non-collapsing and $m\geq 3$), we use the H\"older inequality to get
	\[
		\norm{\tilde{u}_r -u_a}_{L^{\frac{2m}{m-2}}(B_r)} \leq C \norm{\triangle_g u_a-1}_{L^{\frac{2m}{m+2}}(B_r)}\leq C.
	\]
	Here in the last inequality above, we have used \eqref{eqn:goodua}. Note that the above estimate holds uniformly and the constant $C$ is independent of $r$. Together with \eqref{eqn:gooddiff}, we obtain uniform higher order estimate on any compact set for $\tilde{u}_r$ so that Claim 2 holds.
\end{proof}

\subsection{Compatible with the known definition in dimension four}
\label{sub:same}
Finally, when $m=4$, we show that the renormalized volume defined here is the same as the one defined in \cite{biquard2023renormalized}. In fact, it suffices to show the CMC coordinate that was used to define the renormalized volume there satisfies the assumptions of Theorem \ref{thm:main2}.

In Theorem B of \cite{biquard2023renormalized}, there is a coordinate in which
\[
	g= g_e + h_0 + O(\abs{x}^{-5})
\]
where $h_0$ is some symmetric $2$-tensor defined on $\Real^4\setminus \set{0}$ satisfying
\begin{equation}
	\label{eqn:all}
	\mathcal L_{r\partial_r} h_0= -2h_0, \partial_r \llcorner h_0=0, {\mathrm tr}_{g_e}h_0=0, {\rm div}_{g_e}h_0=0, \triangle_{g_e}h_0=0.
\end{equation}
In Proposition 2.6 therein, the dimension of all possible $h_0$ satisfying the above conditions was proved to be $10$. It is well known that the dimension of the space of Weyl tensors ($m=4$) is also $10$. Moreover, it is easy to verify directly that if
\[
	h_0= W_{iabj}\frac{x_ax_b}{\abs{x}^6}
\]
for some $W\in \mathcal W$, then \eqref{eqn:all} holds. Therefore, the CMC coordinate satisfies the assumption of Theorem \ref{thm:main2} and hence defines the same renormalized volume.

\bibliographystyle{alpha}
\bibliography{foo}

\end{document}